\documentclass[
reqno]{amsart} 
\usepackage[T1]{fontenc}
\usepackage[english]{babel}
\usepackage{amssymb,bbm,enumerate}
\usepackage{color}
\usepackage[linktocpage=true,colorlinks=true, linkcolor=blue, citecolor=red, urlcolor=green]{hyperref}

\usepackage{xypic}

\usepackage{tikz-cd}

\renewcommand{\phi}{\varphi}
\renewcommand{\ker}{\Ker}
\def\ch{\mathrm{ch}}
\def\cl{\mathrm{cl}}

\def\al{\alpha}
\def\be{\beta}

\def\ga{\gamma}

\def\la{\lambda}
\def\La{\Lambda}

\newcommand{\mc}[1]{\mathcal{#1}}
\newcommand{\mf}[1]{\mathfrak{#1}}
\newcommand{\mb}[1]{\mathbb{#1}}

\newcommand{\id}{\mathrm{id}} 
\newcommand{\tint}{{\textstyle\int}}

\DeclareMathOperator{\Hom}{Hom}
\DeclareMathOperator{\End}{End}

\DeclareMathOperator{\ad}{ad}
\DeclareMathOperator{\rank}{rank}

\DeclareMathOperator{\Der}{Der}
\DeclareMathOperator{\cent}{Z}
\DeclareMathOperator{\Inder}{Inder}
\DeclareMathOperator{\Cas}{Cas}
\DeclareMathOperator{\Ker}{Ker}

\DeclareMathOperator{\gr}{gr}
\DeclareMathOperator{\cur}{Cur}

\newcommand{\vir}{\mathrm{Vir}}

\newcommand{\LC}{\mathop{\rm LC }}
\newcommand{\PV}{\mathop{\rm PV }}

\newcommand{\PVir}{\mc Vir} 
\newcommand{\Vir}{\mathop{\rm Vir }{}}

\makeatletter
\def\smallunderbrace#1{\mathop{\vtop{\m@th\ialign{##\crcr
   $\hfil\displaystyle{#1}\hfil$\crcr
   \noalign{\kern3\p@\nointerlineskip}%
   \tiny\upbracefill\crcr\noalign{\kern3\p@}}}}\limits}

\theoremstyle{plain}
\newtheorem{theorem}{Theorem}[section]
\newtheorem{lemma}[theorem]{Lemma}
\newtheorem{proposition}[theorem]{Proposition}
\newtheorem{corollary}[theorem]{Corollary}

\theoremstyle{definition}
\newtheorem{definition}[theorem]{Definition}
\newtheorem{example}[theorem]{Example}

\theoremstyle{remark}
\newtheorem{remark}[theorem]{Remark}

\numberwithin{equation}{section}

\definecolor{light}{gray}{.9}

\begin{document}

\title[Computation of LCA and PVA cohomology]{Computation of cohomology of Lie conformal and Poisson vertex algebras}

\author{Bojko Bakalov}
\address{Department of Mathematics, North Carolina State University,
Raleigh, NC 27695, USA}
\email{bojko\_bakalov@ncsu.edu}
\author{Alberto De Sole}
\address{Dipartimento di Matematica, Sapienza Universit\`a di Roma,
P.le Aldo Moro 2, 00185 Rome, Italy}
\email{desole@mat.uniroma1.it}
\urladdr{www1.mat.uniroma1.it/$\sim$desole}
\author{Victor G. Kac}
\address{Department of Mathematics, MIT,
77 Massachusetts Ave., Cambridge, MA 02139, USA}
\email{kac@math.mit.edu}
\subjclass[2010]{
Primary
17B69, 
Secondary 
17B63; 
17B56} 


\begin{abstract}
We develop methods for computation of Poisson vertex algebra cohomology.
This cohomology is computed for the free bosonic and fermionic Poisson vertex (super)algebras, as well as for the universal affine and Virasoro Poisson vertex algebras.
We establish finite dimensionality of this cohomology for conformal Poisson vertex (super)algebras that are finitely and freely generated by elements of positive conformal weight.
\end{abstract}
\keywords{
Lie conformal (super)algebras, 
Poisson vertex (super)algebras, 
affine Lie algebras, 
Virasoro algebra, 
basic cohomology, 
LCA cohomology, 
variational PVA cohomology, 
energy operator.
}

\maketitle


\pagestyle{plain}


\section{Introduction}\label{sec:1}

In the papers \cite{BDSHK18,BDSHK19,BDSHKV19},
we laid down, with our collaborators,
the foundations of the cohomology theory of vertex algebras.
Recall that, 
to any linear symmetric (super)operad $\mc P$ over a field $\mb F$,
one canonically associates a $\mb Z$-graded Lie superalgebra 
\begin{equation}\label{eq:1.1}
W_{\mc P}=\bigoplus_{k=-1}^\infty W^k_{\mc P}
\,,\quad\text{ where }\,\,
W^k_{\mc P}=\mc P(k+1)^{S_{k+1}}
\,.
\end{equation}
The Lie bracket of $W_{\mc P}$ is defined via the $\circ_i$-products of the operad $\mc P$,
see \cite{Tam02} or \cite{BDSHK18} for details.
An odd element $X\in W^1_{\mc P}$ satisfying $[X,X]=0$
defines a cohomology complex $(W_{\mc P},\ad X)$,
which is a differential graded Lie superalgebra.

The most well-known example of this construction is the Lie (super)algebra cohomology.
In this case one takes the operad $\mc Hom(V)$,
for which $\mc Hom(V)(n)=\Hom(V^{\otimes n},V)$,
where $V$ is a fixed vector superspace,
with the action of $S_n$ permuting the factors of $V^{\otimes n}$,
and the well-known $\circ_i$-products,
see e.g.\ \cite{BDSHK18}.
Then $W_{\mc Hom(V)}$ is the Lie superalgebra of polynomial vector fields on $V$.
Furthermore, odd elements $X\in W^1_{\mc Hom(\Pi V)}$,
where $\Pi$ stands for reversing the parity,
such that $[X,X]=0$,
correspond bijectively to Lie superalgebra structures on $V$,
by letting
\begin{equation}\label{eq:1.2}
[a,b]=(-1)^{p(a)}X(a\otimes b)
\,,\qquad a,b\in V
\,.
\end{equation}
The complex $(W_{\mc Hom(\Pi V)},\ad X)$
is then the Chevalley--Eilenberg
cohomology complex of the Lie superalgebra \eqref{eq:1.2}
with coefficients in the adjoint module.
Moreover, given a $V$-module $M$,
we extend the Lie superalgebra structure on $V$
to $V\oplus M$ by making $M$ to be an abelian ideal.
Then the natural reduction of the complex $\big(W_{\mc Hom(\Pi(V\oplus M))},\ad X\big)$
produces the Chevalley--Eilenberg cohomology complex of $V$
with coefficients in $M$, see e.g.\ \cite{DSK13}.
Note that, although the cohomology of $V$ with coefficients in its adjoint module inherits
the Lie superalgebra structure from $W_{\mc Hom(\Pi V)}$,
this is not the case for the reduction.

The next example is the Lie conformal superalgebra cohomology
developed in \cite{BKV99,DSK09,DSK13}.
In this case, one considers the operad $\mc Chom(V)$,
where $V$ is a vector superspace with an even endomorphism $\partial$.
Introduce the vector superspaces
\begin{equation}\label{eq:1.3}
V_n=V[\lambda_1,\dots,\lambda_n]/\langle\partial+\lambda_1+\dots+\lambda_n\rangle
\,,
\end{equation}
where all $\lambda_i$ have even parity and $\langle\Phi\rangle$
stands for the image of the endomorphism $\Phi$.
Then 
\begin{equation}\label{eq:1.4}
\mc Chom(V)(n)
\subset\Hom_{\mb F}(V^{\otimes n},V_n)
\end{equation}
consists of all maps $Y_{\lambda_1,\dots,\lambda_n}\colon V^{\otimes n}\to V_n$
satisfying the sesquilinearity property ($1\leq i\leq n$):
\begin{equation}\label{eq:1.5}
Y_{\lambda_1,\dots,\lambda_n}(v_1\otimes\dots\otimes\partial v_i \otimes\dots\otimes v_n)
=
-\lambda_i \,
Y_{\lambda_1,\dots,\lambda_n}(v_1\otimes\dots\otimes v_n)
\,.
\end{equation}
The action of $S_n$ on $\mc Chom(V)(n)$
is given by the simultaneous permutation of the factors of $V^{\otimes n}$
and the $\lambda_i$'s.
The construction of the products $\circ_i$
can be found in \cite{BDSHK18}.

Then 
odd elements $X\in W^1_{\mc Chom(\Pi V)}$ bijectively correspond
to \emph{skewsymmetric} $\lambda$-\emph{brackets} on $V$,
i.e., maps $[\cdot\,_\lambda\,\cdot]\colon  V^{\otimes2}\to V[\lambda]$
satisfying sesquilinearity
\begin{equation}\label{eq:1.7}
[\partial a_\lambda b]=-\lambda[a_\lambda b]
\,\,,\qquad
[a_\lambda\partial b]=(\partial+\lambda)[a_\lambda b]
\,,
\end{equation}
and skewsymmetry
\begin{equation}\label{eq:1.8}
[a_\lambda b]=-(-1)^{p(a)p(b)}[b_{-\lambda-\partial}a]
\,.
\end{equation}
Explicitly, this bijection is given by
\begin{equation}\label{eq:1.6}
[a_\lambda b]=(-1)^{p(a)}X_{\lambda,-\lambda-\partial}(a\otimes b)
\,.
\end{equation}
Finally, the condition $[X,X]=0$ is equivalent to the Jacobi identity
\begin{equation}\label{eq:1.9}
[a_\lambda[b_\mu c]]
-(-1)^{p(a)p(b)}
[b_\mu[a_\lambda c]]
=
[[a_\lambda b]_{\lambda+\mu}c]
\,.
\end{equation}

Recall that an $\mb F[\partial]$-module $V$,
endowed with a map $V\otimes V\to V[\lambda]$, $a\otimes b\mapsto[a_\lambda b]$,
satisfying conditions \eqref{eq:1.7}, \eqref{eq:1.8}, \eqref{eq:1.9},
is called a \emph{Lie conformal superalgebra} (LCA) \cite{K96}.
Thus, taking for $X\in W^1_{\mc Chom(\Pi V)}$ the map corresponding to the LCA structure on $V$
defined by \eqref{eq:1.6},
we obtain the cohomology complex $(W_{\mc Chom(\Pi V)},\ad X)$,
with the structure of a differential graded Lie superalgebra.
The cohomology of this complex is the LCA cohomology complex with coefficients in the adjoint
module.
By a reduction, mentioned above,
one defines the LCA cohomology complex of $V$ with coefficients in an arbitrary $V$-module.

Yet another important to us example 
is the \emph{variational Poisson vertex (super)algebra} (PVA) cohomology \cite{DSK13}.
Recall that a PVA $\mc V$ is a vector superspace with an even endomorphism $\partial$,
equipped with a structure of a unital commutative associative differential superalgebra,
and a structure of an LCA,
such that the Leibniz rule holds:
\begin{equation}\label{eq:1.10}
[a_\lambda bc]
=
[a_\lambda b]c+(-1)^{p(a)p(b)}b[a_\lambda c]
\,,\qquad a,b,c\in\mc V
\,.
\end{equation}
The variational PVA cohomology complex is constructed for a unital commutative associative differential
superalgebra $\mc V$
by considering the subalgebra 
\begin{equation}\label{eq:1.11}
W_{\PV}(\Pi \mc V)
=
\bigoplus_{k=-1}^\infty W^k_{\PV}(\Pi \mc V)
\end{equation}
of the Lie superalgebra $W_{\mc Chom(\Pi\mc V)}$,
consisting of all maps $Y$ satisfying, besides the sesquilinearity property \eqref{eq:1.5}
and the $S_{k+1}$-invariance,
the Leibniz rule \eqref{eq:leib} below.
Then odd elements $X\in W^1_{\PV}(\Pi \mc V)$
correspond bijectively via \eqref{eq:1.6}
to skewsymmetric $\lambda$-brackets on $V$ satisfying the Leibniz rule \eqref{eq:1.10}.
The condition $[X,X]=0$ is again equivalent to the Jacobi identity \eqref{eq:1.9};
hence such $X$ correspond bijectively to PVA structures on the differential algebra $\mc V$.
The resulting complex $(W_{\PV}(\Pi \mc V),\ad X)$
is called the \emph{variational PVA cohomology complex} of $\mc V$
with coefficients in the adjoint module.
As explained above, given a $\mc V$-module $M$
one defines the corresponding variational PVA cohomology complex
with coefficients in $M$
by a simple reduction procedure.
The corresponding cohomology is denoted by
\begin{equation}\label{eq:1.12}
H_{\PV}(\mc V,M)
=
\bigoplus_{n=0}^\infty 
H^n_{\PV}(\mc V,M)
\,.
\end{equation}
We shift the indices by $1$ as compared with \eqref{eq:1.11}
in order to keep the traditional notation.

The main motivation for the present paper is the computation 
of the vertex algebra cohomology introduced in \cite{BDSHK18}.
It is defined by considering the operad $\mc P_{\ch}(V)$,
which is a local version of the chiral operad of Beilinson and Drinfeld \cite{BD04},
associated to a $\mc D$-module on a smooth algebraic curve $X$,
in the case when $X=\mb F$ and the $\mc D$-module is translation equivariant.
We showed that in this case the operad $\mc P_{\ch}(V)$ 
admits a simple description,
which is an enhancement of the operad $\mc Chom(V)$ described above.

In order to describe this construction,
let $\mc O^{\star,T}_n=\mb F[z_i-z_j,(z_i-z_j)^{-1}]_{1\leq i<j\leq n}$.
For a vector superspace $V$ with an even derivation $\partial$,
the superspace $\mc P_{\ch}(V)(n)$
is defined as the set of all linear maps
\begin{equation}\label{eq:1.13}
Y\colon V^{\otimes n}\otimes\mc O^{\star,T}_n\to V_n
\,\,,\qquad
v_1\otimes\dots\otimes v_n\otimes f\mapsto
Y_{\lambda_1,\dots,\lambda_n}(v_1\otimes\dots\otimes v_n\otimes f)
\,,
\end{equation}
satisfying the following two sesquilinearity properties ($1\leq i\leq n$):
\begin{equation}\label{eq:1.14}
Y_{\lambda_1,\dots,\lambda_n}(v_1\otimes\dots\otimes (\partial+\lambda_i)v_i \otimes\dots\otimes v_n\otimes f)
=
Y_{\lambda_1,\dots,\lambda_n}\Bigl(v_1\otimes\dots\otimes v_n\otimes \frac{\partial f}{\partial z_i}\Bigr)
\,,
\end{equation}
and
\begin{equation}\label{eq:1.15}
Y_{\lambda_1,\dots,\lambda_n}(v_1\otimes\dots\otimes v_n\otimes (z_i-z_j)f)
=
\Bigl(
\frac{\partial}{\partial\lambda_j}
-
\frac{\partial}{\partial\lambda_i}
\Bigr)
Y_{\lambda_1,\dots,\lambda_n}(v_1\otimes\dots\otimes v_n\otimes f)
\,.
\end{equation}
(Note that \eqref{eq:1.14} turns into \eqref{eq:1.5} if $f=1$.)
In \cite{BDSHK18} we also defined the action of $S_n$ on $\mc P_{\ch}(V)(n)$
and the $\circ_i$-products, making $\mc P_{\ch}(V)$ an operad.

As a result,
we obtain the Lie superalgebra 
$$
W_{\ch}(V)=W_{\mc P_{\ch}(V)}=\bigoplus_{k=-1}^\infty W_{\ch}^k(V)
\,,
$$
see \eqref{eq:1.1}.
We show in \cite{BDSHK18}
that odd elements $X\in W_{\ch}^1(\Pi V)$
such that $[X,X]=0$
correspond bijectively to vertex algebra structures on the $\mb F[\partial]$-module $V$,
such that $\partial$ is the translation operator.
As before, this leads to the vertex algebra cohomology
$$
H_{\ch}(V,M)=\bigoplus_{n=0}^\infty H_{\ch}^n(V,M)
\,,
$$ 
for any $V$-module $M$.

Now suppose that the $\mb F[\partial]$-module $V$
is equipped with an increasing $\mb Z_+$-filtration by $\mb F[\partial]$-submodules.
Taking the increasing filtration of $\mc O^{\star,T}_n$ by the number of divisors,
we obtain an increasing filtration of $V^{\otimes n}\otimes\mc O^{\star,T}_n$.
This filtration induces a decreasing filtration of the superspace $\mc P_{\ch}(V)(n)$.
The associated graded spaces $\gr\mc P_{\ch}(V)(n)$
form a graded operad.

On the other hand, in \cite{BDSHK18}
we introduced the closely related operad $\mc P_{\cl}(V)$,
which ``governs'' the Poisson vertex algebra structures on the $\mb F[\partial]$-module $V$.
The vector superspace $\mc P_{\cl}(V)(n)$
is the space of linear maps (cf. \eqref{eq:1.13})
\begin{equation}\label{eq:1.16}
Y\colon \mb F\mc G(n)\otimes V^{\otimes n}
\to V_n \,,
\qquad \Gamma\otimes v \mapsto Y^\Gamma(v)
\,,
\end{equation}
where $\mb F\mc G(n)$ is the vector superspace with even parity 
spanned by oriented graphs with $n$ vertices,
subject to certain conditions.
The corresponding $\mb Z$-graded Lie superalgebra 
$W_{\cl}(\Pi V)=\bigoplus_{k=-1}^\infty W_{\cl}^k(\Pi V)$
is such that odd elements $X\in W_{\cl}^1(\Pi V)$ with $[X,X]=0$
parameterize the PVA structures on the $\mb F[\partial]$-module $V$
by (cf.\ \eqref{eq:1.6}):
\begin{equation}\label{eq:1.17}
ab=(-1)^{p(a)}X^{\bullet\to\bullet}(a\otimes b) 
\,,\qquad
[a_\lambda b]=
(-1)^{p(a)}X^{\bullet\,\,\,\bullet}_{\lambda,-\lambda-\partial}(a\otimes b)
\,.
\end{equation}
This leads to the definition of the classical PVA cohomology.

Assuming that $V$ is endowed with an increasing $\mb Z_+$-filtration by $\mb F[\partial]$-submodules,
we have a canonical linear map of graded operads
\begin{equation}\label{eq:1.18}
\gr\mc P_{\ch}(V)
\to
\mc P_{\cl}(\gr V)
\,.
\end{equation}
We proved in \cite{BDSHK18}
that the map \eqref{eq:1.18} is injective.
The main result of \cite{BDSHK19} is that this map is an isomorphism provided that the filtration of $V$
is induced by a grading by $\mb F[\partial]$-modules.
If, in addition, this filtration of $V$ is such that $\gr V$ inherits from the vertex algebra structure of $V$ a PVA structure, and $\gr M$ inherits a structure of a PVA module over $\gr V$ (see \cite{Li04, DSK05}),
this allows us to compare the vertex algebra cohomology and the classical PVA cohomology 
via a spectral sequence \cite{BDSK19}.

Finally, the obvious inclusion of Lie superalgebras 
$W_{\PV}(\Pi \mc V) \hookrightarrow W_{\cl}(\Pi \mc V)$
induces an injective map in cohomology, and 
we prove in \cite{BDSHKV19} that this map is
an isomorphism between the variational PVA cohomology and the classical PVA cohomology
if, as a differential algebra, 
$\mc V$ is an algebra of differential polynomials. 
This allows us to relate the vertex algebra cohomology $H_{\ch}(V,M)$ 
to the variational PVA cohomology $H_{\PV}(\gr V,\gr M)$ provided that, as a differential algebra, 
$\gr V$ is an algebra of differential polynomials.

This is one of our motivations to study and compute variational PVA cohomology. Another motivation comes from the theory of integrable systems of Hamiltonian PDE. As explained in the introduction to \cite{DSK13}, given a differential algebra $\mc V$ endowed with two compatible PVA structures, the Lenard--Magri scheme of integrability can be infinitely extended, provided that $H^1_{\PV}(\mc V,\mc V)=0$ for one of the PVA structures. Furthermore, the whole cohomology $H_{\PV}(\mc V,\mc V)$ is computed in \cite{DSK13} in the case when the PVA structure is ``quasiconstant.''
The main tool for this computation is the \emph{basic complex}, which is a covering complex of the complex $W_{\PV}(\Pi \mc V)$. The cohomology of the basic complex is easier to compute; then, using the cohomology long exact sequence, one derives information on the cohomology in question. This idea has been utilized already in \cite{BKV99,BDSK09,DSK09}.

In the present paper, we use this idea to compute the variational PVA cohomology 
of the most important examples of PVA's arising in conformal field theory.
All these examples are conformal PVA's, i.e., there exists a Virasoro element $L$,
so that 
$$
[L_\lambda L]=(\partial+2\lambda)L+\frac{c}{12}\lambda^3\,,
$$
for some $c\in\mb F$ (called the central charge),
with the properties that 
$$
[L_\lambda\,\cdot\,]|_{\lambda=0}=\partial
\quad\text{and}\quad
\frac{d}{d\lambda}[L_\lambda\,\cdot\,]|_{\lambda=0} \;\text{ is diagonalizable.}
$$
This allows us to construct the diagonalizable energy operators $\widetilde E$ on the basic complex
and $E$ on the variational PVA complex, compatible with the maps of the long exact sequence connecting them.
Based on this, we prove the main theorem of the paper (Theorem \ref{thm:Delta})
stating that the eigenvalues of the energy operator $E$ on $H_{\PV}(\mc V,M)$
can be only $0$ or $1$,
provided that the conformal PVA $\mc V$, as a differential algebra, is an algebra of differential polynomials.
This theorem, along with formula \eqref{eq:DY} for the eigenvalues of $E$
(conformal weights),
puts stringent conditions on the variational PVA cohomology.

Using this, we prove, for example, that the variational PVA cohomology of free fermions $\mc F_{\mf h}$
is trivial (Theorem \ref{thm:coh-fer}):
\begin{equation}\label{eq:1.21}
\dim H^n_{\PV}(\mc F_{\mf h},\mc F_{\mf h})=\delta_{n,0}
\,,\qquad n\geq0\,.
\end{equation}
Another result is the computation of the variational PVA cohomology of the affine PVA $\mc V^k_{\mf g}$
of nonzero level $k\in\mb F$,
associated to an arbitrary finite-dimensional Lie algebra $\mf g$
with a non-degenerate invariant symmetric bilinear form (Theorem \ref{thm:coh-aff}):
\begin{equation}\label{eq:1.22}
H^n_{\PV}(\mc V^k_{\mf g},\mc V^k_{\mf g})
\simeq
H^n(\mf g,\mb F)\oplus H^{n+1}(\mf g,\mb F)
\,,\qquad n\geq0\,.
\end{equation}
In particular, taking for $\mf g$ an abelian Lie algebra,
we recover the variational PVA cohomology of the free boson,
described separately for pedagogical reasons in Theorem \ref{thm:coh-bos},
and previously computed in \cite{DSK12}.

In a similar way, we compute the variational PVA cohomology of the Virasoro PVA $\PVir^c$ for any central charge $c\in\mb F$
(Theorem \ref{thm:coh-vir}):
\begin{equation}\label{eq:1.23}
\dim H_{\PV}^n(\PVir^c, \PVir^c) 
= 
\begin{cases} \, 1 \,, \quad\text{for } \; n=0,2,3, 
\\ \, 0 \,, \quad\;\;\; \text{otherwise.}
\end{cases}
\end{equation}
Note that our Theorem \ref{prop}
relates the cohomology of an LCA $R$ and the variational PVA cohomology of the associated PVA $S(R)$.
The LCA cohomology of the main examples was computed already in \cite{BKV99}
for centerless LCA's, and in Section \ref{sec:2} we derive it for the central extensions
using Proposition \ref{prop:central-ext}.
This is used in the proof of Theorem \ref{thm:coh-vir}.

The interpretation of the LCA cohomology and the variational PVA cohomology
in degrees $0$, $1$ and $2$,
given by Theorems \ref{thm:lowcoho} and \ref{thm:lowcoho2} respectively,
allows us to compute the Casimirs, derivations and first-order deformations.
For example, formulas \eqref{eq:1.21}, \eqref{eq:1.22} and \eqref{eq:1.23}
show that the Casimirs are trivial for the PVA's $\mc F_{\mf h}$, $\mc V^k_{\mf g}$
with $k\neq0$ and $\mf g$ simple,
and for $\PVir^c$ for any $c$;
all their derivations are innner;
the first-order deformations are trivial for $\mc F_{\mf h}$
and are the obvious ones for $\mc V^k_{\mf g}$ with $k\neq0$ and $\mf g$ simple,
and for $\PVir^c$.


From Theorem \ref{thm:Delta} and formula \eqref{eq:DY},
we deduce Theorem \ref{thm:Delta2},
which states that for a conformal PVA $\mc V$ which,
as a differential algebra,
is an algebra of differential polynomials on generators of positive conformal weight,
all cohomology spaces $H^n_{\PV}(\mc V,\mc V)$
are finite dimensional.

In our next paper \cite{BDSK19}, we address the problem of exact
computation of cohomology of freely generated conformal vertex algebras.

Throughout the paper, the base field $\mb F$ is a field of characteristic $0$,
and, unless otherwise specified, all vector (super)spaces, their tensor products and Hom's 
are over $\mb F$; the parity of a vector superspace will be denoted by $p$.
We denote by $\mb Z_+$ the set of non-negative integers.

\subsubsection*{Acknowledgments} 

This research was partially conducted during the authors' visits
to the University of Rome La Sapienza and to MIT.
The first author was supported in part by a Simons Foundation grant 584741.
The second author was partially supported by the national PRIN fund n. 2015ZWST2C$\_$001
and the University funds n. RM116154CB35DFD3 and RM11715C7FB74D63.
All three authors were supported in part by the Bert and Ann Kostant fund.

\section{Lie conformal algebra cohomology}
\label{sec:2}
In this section, we review the definitions of a Lie conformal superalgebra, a module over it, and its cohomology. We also calculate the cohomology in the main examples, based on the results of \cite{BKV99}.

\subsection{Lie conformal algebras}
\label{sec:2.1}
In this subsection, we recall the definitions of a Lie conformal superalgebra (henceforth abbreviated LCA) and a module over it \cite{K96, DAK98}. We also recall the main examples.

\begin{definition}\label{def:lca}
Let $R$ be a vector superspace with parity $p$,
endowed with an even endomorphism $\partial$.
A \emph{Lie conformal superalgebra} (LCA) structure on $R$
is a bilinear, parity preserving $\lambda$-\emph{bracket}
$R\otimes R\to R[\lambda]$, $a\otimes b\mapsto[a_\lambda b]$,
satisfying ($a,b,c\in R$):
\begin{enumerate}[L1\,]
\item
$[\partial a_\lambda b]=-\lambda[a_\lambda b]$, \;
$[a_\lambda\partial b]=(\lambda+\partial)[a_\lambda b]$ \;
(sesquilinearity);
\item
$[a_\lambda b]=-(-1)^{p(a)p(b)}[b_{-\lambda-\partial}a]$ \;
(skewsymmetry);
\item
$[a_\lambda[b_\mu c]]-(-1)^{p(a)p(b)}[b_\mu[a_\lambda c]]
=
[[a_\lambda b]_{\lambda+\mu}c]$ \;
(Jacobi identity).
\end{enumerate}
A \emph{module} over the LCA $R$
is a vector superspace $M$ with an even endomorphism $\partial$,
endowed with a bilinear, parity preserving $\lambda$-action 
$R\otimes M\to M[\lambda]$, $a\otimes m\mapsto a_\lambda m$,
satisfying ($a,b\in R$, $m\in M$):
\begin{enumerate}[M1]
\item
$(\partial a)_\lambda m=-\lambda a_\lambda m$, \;
$a_\lambda(\partial m)=(\lambda+\partial)(a_\lambda m)$;
\item
$a_\lambda(b_\mu m)-(-1)^{p(a)p(b)}b_\mu(a_\lambda m)
=
[a_\lambda b]_{\lambda+\mu}m$.
\end{enumerate}
\end{definition}

Throughout the paper, for an $\mb F[\partial]$-module $R$, we will denote by $\tint\colon R\to R/\partial R$ the canonical quotient map. Recall that when $R$ is an LCA,
the vector superspace $R/\partial R$ carries a canonical Lie superalgebra structure with bracket
$$
\bigl[ \tint a, \tint b \bigr] = \tint [a_\la b] |_{\la=0} 
\,, \qquad a,b \in R \,.
$$
Moreover, any $R$-module $M$ has the structure of a Lie superalgebra module over $R/\partial R$ given by
$$
\bigl( \tint a \bigr) m = a_\la m |_{\la=0}
\,, \qquad a \in R \,, \; m\in M \,.
$$

\begin{example}[Free superboson LCA]\label{ex:boson-lca}
Let $\mf h$ be a finite-dimensional superspace, with parity $p$, and a supersymmetric nondegenerate bilinear form $(\cdot|\cdot)$. By supersymmetry of the form we mean that $(a|b)=(-1)^{p(a)p(b)} (b|a)$ for $a,b\in\mf h$
and $(a|b)=0$ whenever $p(a)\ne p(b)$.
The \emph{free superboson} LCA corresponding to $\mf h$ is the $\mb F[\partial]$-module
$$
R^b_{\mf h}=\mb F[\partial]\mf h\oplus\mb FK
\,,\qquad\text{where}\quad
\partial K=0 \,, \;\; p(K)=\bar0
\,,
$$
endowed with the $\lambda$-bracket
\begin{equation}\label{eq:boson}
[a_\lambda b]=\lambda (a|b) K
\,\,\text{ for }\,\, a,b\in\mf h\,,\qquad
K \;\text{ central}
\end{equation}
(uniquely extended to $R^b_{\mf h}\otimes R^b_{\mf h}$ by the sesquilinearity axioms). 
In the case when $\mf h$ is purely even, i.e., $p(a)=\bar0$ for all $a\in\mf h$, 
the LCA $R^b_{\mf h}$ is called the \emph{free boson} LCA.
\end{example}
\begin{example}[Free superfermion LCA]\label{ex:fermion-lca}
Let $\mf h$ be a finite-dimensional superspace, with parity $p$, and a super-skewsymmetric nondegenerate bilinear form $(\cdot|\cdot)$. Now we have $(a|b)=-(-1)^{p(a)p(b)} (b|a)$ for $a,b\in\mf h$
and $(a|b)=0$ whenever $p(a)\ne p(b)$.
The \emph{free superfermion} LCA corresponding to $\mf h$ is the $\mb F[\partial]$-module
$$
R^f_{\mf h}=\mb F[\partial]\mf h\oplus\mb F K
\,,\qquad\text{where}\quad
\partial K=0 \,, \;\; p(K)=0
\,,
$$
endowed with the $\lambda$-bracket
\begin{equation}\label{eq:fermion}
[a_\lambda b]=(a|b) K
\,,\qquad
K \;\text{ central}
\end{equation}
(uniquely extended to $R^f_{\mf h}\otimes R^f_{\mf h}$ by the sesquilinearity axioms). 
In the case when $p(a)=\bar1$ for all $a\in\mf h$, the LCA $R^f_{\mf h}$ is called the \emph{free fermion} LCA.
\end{example}
\begin{example}[Affine LCA]\label{ex:affine-lca}
Let $\mf g$ be a Lie algebra with a nondegenerate invariant symmetric bilinear form $(\cdot\,|\,\cdot)$.
The corresponding \emph{affine} LCA is the purely even $\mb F[\partial]$-module
$$
\cur\mf g=\mb F[\partial]\mf g\oplus\mb FK
\,,\;\;
\text{ where }\;
\partial K=0
\,,
$$
endowed with the $\lambda$-bracket
given on the generators by
\begin{equation}\label{eq:current}
[a_\lambda b]=[a,b]+\lambda (a|b) K
\,,\qquad
a,b\in\mf g
\,,\quad
K \text{ central.}
\end{equation}
\end{example}
\begin{example}[Virasoro LCA]\label{ex:virasoro-lca}
The \emph{Virasoro} LCA is the purely even $\mb F[\partial]$-module
$$
R^{\vir}=\mb F[\partial]L\oplus\mb FC
\,,\;\;
\text{ where }\;
\partial C=0
\,,
$$
endowed with the $\lambda$-bracket
\begin{equation}\label{eq:vir}
[L_\lambda L]=(\partial+2\lambda)L+\frac{1}{12}\lambda^3 C
\,,\qquad
C \text{ central.}
\end{equation}
\end{example}

The importance of the last two examples stems from the fact that the LCA's $\overline{\cur}\,\mf g = \cur\mf g / \mb F K$ for $\mf g$ simple and $\bar R^{\vir} = R^{\vir} / \mb F C$ exhaust all simple LCA's, which are finitely generated as $\mb F[\partial]$-modules \cite{DAK98}.

\subsection{LCA cohomology}
\label{sec:2.2}

Throughout the paper, we shall use the following notation:
for a vector superspace with parity $p$,
we denote by $\bar p=1-p$ the opposite parity.
Given a module $M$ over the LCA $R$,
the corresponding cohomology complex is constructed as follows \cite{BKV99, DSK13}.
For $n\geq0$, an $n$-\emph{cochain} of $R$ with coefficients in $M$ 
is a linear map
\begin{equation}\label{eq:lca-maps}
Y\colon
R^{\otimes n}
\longrightarrow
M[\lambda_1,\dots,\lambda_n]/\langle\partial+\lambda_1+\cdots+\lambda_n\rangle
\,,
\end{equation}
where $\langle\Phi\rangle$ denotes the image of the endomorphism $\Phi$,
satisfying the \emph{sesquilinearity} conditions ($1\leq i\leq n$):
\begin{equation}\label{eq:sesq}
Y_{\lambda_1,\dots,\lambda_n}(a_1\otimes\cdots\otimes (\partial a_i) 
\otimes\cdots\otimes a_n)
=
-\lambda_iY_{\lambda_1,\dots,\lambda_n}(a_1\otimes\cdots\otimes a_n)
\,,
\end{equation}
and the \emph{symmetry} conditions ($1\leq i< n$):
\begin{equation}\label{eq:skew}
\begin{split}
& Y_{\lambda_1,\dots,\lambda_i,\lambda_{i+1},\dots,\lambda_n}
(a_1\otimes\cdots\otimes a_i\otimes a_{i+1} \otimes\cdots\otimes a_n) \\
& =
(-1)^{\bar p(a_i)\bar p(a_{i+1})}
Y_{\lambda_1,\dots,\lambda_{i+1},\lambda_i,\dots,\lambda_n}
(a_1\otimes\cdots\otimes a_{i+1}\otimes a_i \otimes\cdots\otimes a_n)
\,.
\end{split}
\end{equation}
(Note that \eqref{eq:skew} is indeed a symmetry condition with respect to the parity $\bar p$,
but in the purely even case $p=\bar0$ it is in fact \emph{skewsymmetry}.)

We let $C_{\LC}^n(R,M)$ be the vector superspace of $n$-cochains,
with parity $\bar p$ induced by the parity $\bar p=1-p$ of $R$ and $M$, letting
all $\la_i$ be even.
For example, 
\begin{equation}\label{eq:W-ksmall}
C_{\LC}^{0}(R,M)=M/\partial M
\,,\qquad
C_{\LC}^{1}(R,M)=\Hom_{\mb F[\partial]}(R,M)
\,,
\end{equation}
while $C_{\LC}^2(R,M)$ can be identified with the space of $\lambda$-brackets
$Y\colon R\otimes R\to M[\lambda]$
that satisfy the sesquilinearity L1
and symmetry with respect to $\bar p$
(cf.\ L2).
We let 
\begin{equation}\label{eq:W-partial}
C_{\LC}(R,M)=\bigoplus_{n \geq0}C_{\LC}^n(R,M)
\,.
\end{equation}
Note that in the case when $R$ and $M$ are purely even, 
the parity $\bar p$ of $C_{\LC}^n(R,M)$ is $n-1$ mod $2\mb Z$.

The LCA cohomology differential
$d\colon C_{\LC}^{n}(R,M)\to C_{\LC}^{n+1}(R,M)$, for $n\geq0$,
is defined by (cf.\ \cite[Eq. (4.19)]{DSK13})
\begin{equation}\label{eq:lca-d}
\begin{split}
(dY&)_{\lambda_0,\dots,\lambda_n}
(a_0\otimes\cdots\otimes a_n) 
=
\sum_{i=0}^n
(-1)^{\gamma_i}
{a_i}_{\lambda_i}
Y_{\lambda_0,\stackrel{i}{\check{\dots}},\lambda_n}
(a_0\otimes\stackrel{i}{\check{\dots}}\otimes a_n) 
\\
& +
\sum_{0\leq i<j\leq n}
(-1)^{\gamma_{ij}}
Y_{\lambda_i+\lambda_j,\lambda_0,
\stackrel{i}{\check{\dots}}\stackrel{j}{\check{\dots}},\lambda_n}
([{a_i}_{\lambda_i}a_j]\otimes a_0\otimes
\stackrel{i}{\check{\dots}}\stackrel{j}{\check{\dots}}
\otimes a_n) 
\,,
\end{split}
\end{equation}
where
\begin{equation}\label{eq:gammai}
\begin{split}
\gamma_i &=
\bar p(a_i)
(\bar p(Y)+\bar p(a_0)+\dots+\bar p(a_{i-1})+1)+1
\\
\gamma_{ij} &=
\bar p(Y)+ 
\bar p(a_i) (\bar p(a_0)+\dots+\bar p(a_{i-1})+1)
+ \bar p(a_j) (\bar p(a_0)+\stackrel{i}{\check{\dots}}+\bar p(a_{j-1})\bigr)
\,.
\end{split}
\end{equation}
In particular, if $R$ is purely even, then
\begin{equation}\label{eq:gammai-even}
\gamma_i=n+i+1
\,\,\;\text{ and }\;\,\,
\gamma_{ij}=n+i+j+1
\,,
\end{equation}
(note that formulas \cite[(2.18) and (4.20)]{DSK13} are not quite correct:
the overall factor there should be $(-1)^{n+1}$ instead of $(-1)^n$).
If all $a_i$'s are odd, then
\begin{equation}\label{eq:gammai-odd}
\gamma_i=1
\,\,\;\text{ and }\;\,\,
\gamma_{ij}=\bar p(Y)
\,.
\end{equation}
\begin{proposition}[{\cite{BKV99,BDAK01,DSK09}}]\label{prop:LCAd}
Eq.\ \eqref{eq:lca-d}
defines an odd endomorphism of the vector superspace\/ $C_{\LC}(R,M)$ of degree\/ $1$,
such that\/ $d^2=0$. 
\end{proposition}
\begin{definition}\label{def:lca-coho}
Given a module $M$ over the LCA $R$,
the cohomology of the complex $(C_{\LC}(R,M),d)$
is called the \emph{LCA cohomology} of $R$ with coefficients in $M$:
\begin{equation}\label{eq:h-lca}
\begin{split}
H_{\LC}(R,M)
&=
\bigoplus_{n\geq0}
H_{\LC}^n(R,M)
\,,\\
H_{\LC}^n(R,M)
&=
\ker\big(d|_{C_{\LC}^{n}(R,M)}\big)
/
d\big(C_{\LC}^{n-1}(R,M)\big)
\,.
\end{split}
\end{equation}
\end{definition}
\begin{remark}\label{rem:LCAoperad}
Recall that, for an operad $\mc P$, the subspace $W_{\mc P}\subset\mc P$ of symmetric elements
carries the structure of a Lie superalgebra \cite{Tam02}.
As explained in the introduction and in \cite{BDSHK18},
for $M=R$ the space $\mc Chom(\Pi R)$ of maps \eqref{eq:lca-maps}
satisfying the sesquilinearity conditions \eqref{eq:sesq}
has the structure of an operad,
with parity $\bar p$ induced by the opposite parity of $R$,
with some natural actions of the symmetric groups,
and with some natural composition maps.
The space $W_{\LC}(\Pi R)=\bigoplus_{k\geq-1}W_{\LC}^k(\Pi R)$,
where $W_{\LC}^k(\Pi R)$
is the subspace of symmetric elements 
in $\mc Chom(\Pi R)({k+1})$,
is naturally a $\mb Z$-graded Lie superalgebra.
Note that $W_{\LC}^k(\Pi R)=C_{\LC}^{k+1}(R,R)$.
%
%
LCA structures on $R$ bijectively correspond to odd (with respect to the parity $\bar p$)
elements $X\in W_{\LC}^1(\Pi R)$
such that $[X,X]=0$,
and the cohomology differential \eqref{eq:lca-d} is $d=\ad X$.
%
\end{remark}

\subsection{Low degree cohomology}
\label{sec:2.3}

Let $R$ be an LCA and $M$ be an $R$-module.
A \emph{Casimir element} is an element $\int m\in M/\partial M$
such that  $R_{-\partial}m =0$.
Denote by $\Cas(R,M)\subset M/\partial M$ the space of Casimir elements.
Note that $\Cas(R,R)=\{\int a\in R/\partial R\,|\,[a_{\lambda}R]|_{\lambda=0}=0\}$.

A \emph{derivation} from $R$ to $M$
is an $\mb F[\partial]$-module homomorphism $D\colon R\to M$ such that 
\begin{equation}\label{eq:der}
D[{a}_\lambda{b}]
=
(-1)^{p(D) p(a)}
\,{a}_\lambda{D(b)}
+
(-1)^{1+(p(D)+p(a))p(b)}
\,b_{-\lambda-\partial}{D(a)}
\,,
\end{equation}
for all $a,b \in R$.
We say that a derivation is \emph{inner} if it has the following form:
\begin{equation}\label{eq:inner}
D_{\tint m}(a)=(-1)^{1+p(m)p(a)}a_{-\partial}m
\,\quad\text{ for some }\quad \tint m\in M/\partial M
\,.
\end{equation}
Denote by $\Der(R,M)$ the space of derivations from $R$ to $M$,
and by $\Inder(R,M)$ the subspace of inner derivations.
In the special case when $M=R$,
we have the usual definition of a derivation and an inner derivation of the LCA $R$
as an $\mb F[\partial]$-module endomorphism $D$ such that
\begin{equation}\label{eq:der2}
D[{a}_\lambda{b}]
= [(Da)_\lambda{b}]
+ (-1)^{p(D) p(a)} [{a}_\lambda (Db)]
\,,
\end{equation}
and, respectively,
\begin{equation}\label{eq:inner2}
(\ad a)b = a_{(0)} b = [{a}_\lambda{b}] |_{\la=0}
\,.
\end{equation}

\begin{theorem}[{\cite{BKV99,DSK09}}]\label{thm:lowcoho}
Let\/ $R$ be a Lie conformal algebra, and\/ $M$ be an\/ $R$-module. Then:
\begin{enumerate}[(a)]
\item
$H^0_{\LC}(R,M)=\Cas(R,M)$.
\item
$H^1_{\LC}(R,M)=\Der(R,M)/\Inder(R,M)$.
\item
$H^2_{\LC}(R,M)$ is the space of isomorphism classes of\/ $\mb F[\partial]$-split extensions
of the LCA\/ $R$ by the\/ $R$-module\/ $M$, where\/ $M$ is
viewed as an LCA with zero\/ $\lambda$-bracket.
In particular, $H^2_{\LC}(R,R)$ parameterizes the equivalence classes of first-order deformations of\/ $R$ that preserve the\/ $\mb F[\partial]$-module structure.
\end{enumerate}
\end{theorem}

\subsection{Central extensions}
\label{sec:2.4}

An element $C\in R$ is called \emph{central} if $[C_\lambda R]=0$
(or, equivalently, $[R_\lambda C]=0$).
In particular, $\mb F[\partial]C\subset R$ is an LCA ideal
and one can consider the quotient LCA $\bar{R}=R/\mb F[\partial]C$.
The element $C\in R$ is called \emph{torsion} if $P(\partial)C=0$
for some nonzero polynomial $P(\partial)$.
\begin{lemma}\label{lem:tor}
Let\/ $R$ be an LCA and\/ $M$ be an\/ $R$-module.
Suppose $C\in R$ is a torsion element and consider $\bar{R}=R/\mb F[\partial]C$. Then{\rm{:}}
\begin{enumerate}[(a)]
\item
$C$ acts trivially on any $R$-module $M$. In particular, $C$ is central in $R$.
\item
$\bar R$ is an LCA, and an $R$-module is the same as an $\bar{R}$-module.
\item
For\/ $n\ge2$, any\/ $n$-cochain\/ $Y\in C_{\LC}^n(R,M)$ vanishes when one of its arguments lies in\/
$\mb F[\partial]C$.
\end{enumerate}
\end{lemma}
\begin{proof}
Part (a) follows from the sesquilinearity axiom (cf.\ \cite{DAK98}), and part (b) follows immediately from (a). Similarly,
if $P(\partial)C=0$, by the sesquilinearity condition \eqref{eq:sesq},
we have
\begin{align*}
P(-\lambda_i)
Y_{\lambda_1,\dots,\lambda_n}(a_1\otimes\dots\otimes C \otimes\dots\otimes a_n)
=0 
\,&\in
M[\lambda_1,\dots,\lambda_n]/\langle\partial+\lambda_1+\dots+\lambda_n\rangle
\\
&\simeq
M[\lambda_1,\dots,\lambda_{n-1}]
\,.
\end{align*}
This proves (c).
\end{proof}

\begin{proposition}\label{prop:central-ext}
Let\/ $R$ be an LCA, let\/ $C\in R$ be a torsion element,
and consider the quotient LCA\/ $\bar{R}=R/\mb F[\partial] C$.
Let\/ $M$ be a module over the LCA\/ $R$.
Then:
\begin{enumerate}[(a)]
\item
We have canonical linear maps
\begin{equation}\label{eq:w-maps}
\begin{split}
& C_{\LC}^1(\bar{R},M)\hookrightarrow C_{\LC}^1(R,M)
\,,\\
& C_{\LC}^n(\bar{R},M)\stackrel{\sim}{\longrightarrow}C_{\LC}^n(R,M)
\,\;\text{ for }\; n\neq1
\,.
\end{split}
\end{equation}
\item
We have canonical linear maps
\begin{equation}\label{eq:h-maps}
\begin{split}
& H_{\LC}^1(\bar{R},M)\hookrightarrow H_{\LC}^1(R,M)
\,,\\
& H_{\LC}^2(\bar{R},M)\twoheadrightarrow H_{\LC}^2(R,M)
\,,\\
& H_{\LC}^n(\bar{R},M)\stackrel{\sim}{\longrightarrow}H_{\LC}^n(R,M)
\;\text{ for }\; n\neq1,2
\,.
\end{split}
\end{equation}
\item
Let\/ $P(\partial)\in\mb F[\partial]$ be the minimal monic polynomial that
annihilates\/ $C$, and let
$$
U=\big\{u\in M\,\big|\,P(\partial)u=0\big\}
\,.
$$
If\/ $R$ splits, as an\/ $\mb F[\partial]$-module, as\/ $R\simeq \bar{R}\oplus\mb F[\partial]C$,
then
\begin{equation}\label{eq:H01}
\begin{split}
& \bigl( \dim H_{\LC}^1(R,M)-\dim H_{\LC}^1(\bar{R},M) \bigr) \\
& +\bigl( \dim H_{\LC}^2(\bar{R},M) - \dim H_{\LC}^2(R,M) \bigr)
=
\dim U
\,.
\end{split}
\end{equation}
Note that, in the left-hand side of \eqref{eq:H01}, both summands are non-negative by part (b).
\end{enumerate}
\end{proposition}
\begin{proof}
For every $n\geq0$, we have the canonical injective map
\begin{equation}\label{eq:W-map}
C_{\LC}^n(\bar{R},M)\hookrightarrow C_{\LC}^n(R,M)
\,\,,\quad
\bar Y\mapsto Y=\bar Y\circ\pi
\,,
\end{equation}
obtained by composing with the quotient map 
$\pi\colon R^{\otimes n}\twoheadrightarrow \bar{R}^{\otimes n}$:
$$
\begin{tikzcd}
R^{\otimes n} \arrow[->>]{r}{\pi} \arrow[rr, "{Y}", bend left=30]
& \bar{R}^{\otimes n} \arrow{r}{\bar Y} &\phantom{\Bigg(}&\!\!\!\!\!\!\!\!\!\!\!\!\!\!\!\!\!\!\!\!\!
M[\lambda_1,\dots,\lambda_n]/\langle\partial+\lambda_1+\dots+\lambda_n\rangle \,.
\end{tikzcd}
$$
For $n=0$, this map is obviously a bijection, since 
$C_{\LC}^{0}(R,M)=C_{\LC}^{0}(\bar{R},M)\simeq M/\partial M$.
Let $n\geq1$ and let $Y\in C_{\LC}^{n}(R,M)$.
For $n\geq2$, Lemma \ref{lem:tor}(c) implies that $Y$ factors to a map $\bar Y\in C_{\LC}^n(\bar{R},M)$,
hence \eqref{eq:W-map} is surjective.
For $n=1$, the image of the map \eqref{eq:W-map} is
$$
\big\{Y\in \Hom_{\mb F[\partial]}(R,M)\,\big|\,Y(\mb F[\partial]C)=0\big\}
\subset
\Hom_{\mb F[\partial]}(R,M)
=
C_{\LC}^1(R,M)
\,.
$$
This proves part (a).

It is immediate to check that the action of the differential $d$ given by \eqref{eq:lca-d}
commutes with the map \eqref{eq:W-map}.
Claim (b) is then an obvious consequence of (a).

Finally, we prove (c).
Under the assumption that $R\simeq \bar{R}\oplus\mb F[\partial]C$ as $\mb F[\partial]$-modules,
we have
\begin{equation}\label{0111:eq1}
\begin{split}
C_{\LC}^1&(R,M)
=
\Hom_{\mb F[\partial]}(R,M) \\
& =
\Hom_{\mb F[\partial]}(\bar{R},M)
\oplus
\Hom_{\mb F[\partial]}(\mb F[\partial]C,M)
\simeq
C_{\LC}^1(\bar{R},M)\oplus U
\,.
\end{split}
\end{equation}
By looking at the kernel of $d$ on both sides of \eqref{0111:eq1}, we get
\begin{align*}
& \dim\bigl(\ker d\big|_{C_{\LC}^1(R,M)}\bigr)
=
\dim
\big\{
Y-u\,\big|\,
Y\in C_{\LC}^1(\bar{R},M),\,u\in U,\,dY=du
\big\} \\
& =
\dim\bigl(\ker d\big|_{C_{\LC}^1(\bar{R},M)}\bigr)
+
\dim\bigl(\ker d|_U\bigr)
+
\dim\bigl(d C_{\LC}^1(\bar{R},M)
\cap
d U
\bigr)
\,.
\end{align*}
Hence,
\begin{equation}\label{0111:eq2}
\dim H_{\LC}^1(R,M)-\dim H_{\LC}^1(\bar{R},M)
=
\dim
\ker\bigl(d|_U\bigr)
+
\dim\bigl(
dC_{\LC}^1(\bar{R},M)
\cap
dU
\bigr)
\,.
\end{equation}
By looking at the image of $d$ on both sides of \eqref{0111:eq1}, we get
\begin{align*}
&\dim(d C_{\LC}^1(R,M))
=
\dim\bigl(d C_{\LC}^1(\bar{R},M) + dU\bigr) \\
& =
\dim\bigl(d C_{\LC}^1(\bar{R},M)\bigr) 
+
\dim(dU)
-
\dim\bigl(
d C_{\LC}^1(\bar{R},M)
\cap
d U
\bigr)
\,.
\end{align*}
Hence,
\begin{equation}\label{0111:eq3}
\dim H_{\LC}^2(R,M)-\dim H_{\LC}^2(\bar{R},M)
=
-
\dim(dU)
+
\dim\bigl(
d C_{\LC}^1(\bar{R},M)
\cap
d U
\bigr)
\,.
\end{equation}
Combining equations \eqref{0111:eq2} and \eqref{0111:eq3}, we get \eqref{eq:H01},
thus completing the proof.
\end{proof}

\begin{corollary}\label{cor:central-ext}
Consider an LCA\/ $R$ that splits, as an\/ $\mb F[\partial]$-module, as\/ $R\simeq \bar{R}\oplus\mb F[\partial]C$,
where\/ $C\in R$ is a torsion element. Let\/ $M$ be an\/ $R$-module, which is torsion-free as an\/ $\mb F[\partial]$-module.
Then\/ $H_{\LC}^n(R,M) \simeq H_{\LC}^n(\bar{R},M)$ for all\/ $n\ge 0$.
\end{corollary}

\subsection{Basic LCA cohomology complex}\label{sec:bas-lca}

In this subsection, we review the basic LCA cohomology complex, which was introduced in \cite{BKV99} in the purely even case.
As before, let $R$ be an LCA and $M$ be an $R$-module. 

For $n\geq0$, a \emph{basic} $n$-\emph{cochain} of $R$ with coefficients in $M$ 
is a linear map (cf.\ \eqref{eq:lca-maps})
\begin{equation}\label{eq:lca-maps-b}
\widetilde Y\colon
R^{\otimes n}
\longrightarrow
M[\lambda_1,\dots,\lambda_n]
\,,
\end{equation}
satisfying the sesquilinearity conditions \eqref{eq:sesq}
and the symmetry conditions \eqref{eq:skew}.
We let $\widetilde C_{\LC}^n(R,M)$ be the superspace of basic $n$-cochains,
with parity $\bar p$ induced by the opposite parities $\bar p=1-p$ of $R$ and $M$, and let
\begin{equation}\label{eq:tildeW-lc}
\widetilde C_{\LC}(R,M)=\bigoplus_{n\geq0} \widetilde C_{\LC}^n(R,M)
\,.
\end{equation}

For a basic $n$-cochain $\widetilde{Y}$, we define
\begin{equation}\label{eq:tilde-partial}
(\partial\widetilde{Y})_{\lambda_1,\dots,\lambda_n}
(a_1\otimes\dots\otimes a_n)
=
(\partial+\lambda_1+\dots+\lambda_n)
\widetilde{Y}_{\lambda_1,\dots,\lambda_n}(a_1\otimes\dots\otimes a_n)
\,,
\end{equation}
which is obviously again a basic $n$-cochain. 
Thus, we have even endomorphisms $\partial$ of $\widetilde{C}_{\LC}^n(R,M)$ for all $n\ge0$.

\begin{lemma}[\cite{BKV99}]\label{lem:partial-inj}
The endomorphism\/ $\partial$ of\/ $\widetilde{C}_{\LC}^n(R,M)$ is injective for all\/ $n\ge1$.
\end{lemma}
\begin{proof}
See \cite[Proposition 2.1]{BKV99}.
\end{proof}

Consider the quotient map
\begin{equation}\label{eq:piM}
\pi\colon M[\lambda_1,\dots,\lambda_n]\twoheadrightarrow
M[\lambda_1,\dots,\lambda_n]/\langle\partial+\lambda_1+\dots+\lambda_n\rangle
\,.
\end{equation}
Clearly, if $\widetilde Y$ satisfies equations \eqref{eq:sesq} and \eqref{eq:skew}
in $M[\lambda_1,\dots,\lambda_n]$,
then so does $\pi\circ\widetilde Y$ in
$M[\lambda_1,\dots,\lambda_n]/\langle\partial+\lambda_1+\dots+\lambda_n\rangle$.
Hence, composing with $\pi$
gives a well-defined linear map 
\begin{equation}\label{eq:piYtilde}
\widetilde{C}_{\LC}^n(R,M)\to {C}_{\LC}^n(R,M) \,,
\qquad \widetilde Y\mapsto Y=\pi\circ\widetilde Y \,.
\end{equation}
Since $\pi\circ(\partial\widetilde{Y})=0$, it induces a linear map
\begin{equation}\label{eq:piYtilde2}
\widetilde{C}_{\LC}^n(R,M) / \partial\widetilde{C}_{\LC}^n(R,M) \to {C}_{\LC}^n(R,M) \,,
\qquad \widetilde Y+\langle\partial\rangle\mapsto Y=\pi\circ\widetilde Y \,.
\end{equation}

\begin{lemma}[\cite{DSK13}]\label{lem:Ytilde-lca}
\begin{enumerate}[(a)]
\item
The map \eqref{eq:piYtilde2} is injective for all\/ $n\ge0$.
\item
Suppose that, as an\/ $\mb F[\partial]$-module, $R$ is a direct sum of a torsion module and a free module.
Then the map \eqref{eq:piYtilde2} is surjective for\/ $n=0$ and all\/ $n\ge2$.
\item
If $R$ is free as an\/ $\mb F[\partial]$-module, then the map \eqref{eq:piYtilde2} is surjective for\/ $n=1$ as well.
\end{enumerate}
\end{lemma}
\begin{proof}
See Proposition 6.5 and Remark 6.6 in \cite{DSK13}.
\end{proof}

The basic LCA cohomology \emph{differential} $\widetilde d$ is defined again by \eqref{eq:lca-d},
viewed as an equation in $M[\lambda_0,\dots,\lambda_n]$.
\begin{lemma}\label{lem:basic-d}
\begin{enumerate}[(a)]
\item
Formula \eqref{eq:lca-d}
defines a map 
$\widetilde d\colon \widetilde C_{\LC}^{n}(R,M)\to \widetilde C_{\LC}^{n+1}(R,M)$.
\item
$\widetilde{d}(\partial\widetilde{Y})=\partial(\widetilde{d}\,\widetilde{Y})$.
\item
$\widetilde{d}^2=0$.
\item
$\pi\circ(\widetilde{d}\,\widetilde{Y})=d(\pi\circ\widetilde{Y})$.
\end{enumerate}
\end{lemma}
\begin{proof}
Parts (a), (b) and (c) are proved in \cite[Lemma 2.1]{BKV99} in the even case,
and in \cite{DSK13} in the super case.
Part (d) is obvious.
\end{proof}

\begin{definition}\label{def:lca-bas}
Given a module $M$ over the LCA $R$,
the cohomology of the complex $\bigl(\widetilde C_{\LC}(R,M),\widetilde d\bigr)$
is called the \emph{basic LCA cohomology} of $R$ with coefficients in $M$:
\begin{equation}\label{eq:h-lca-bas}
\begin{split}
\widetilde H_{\LC}(R,M)
&=
\bigoplus_{n\geq0}
\widetilde H_{\LC}^n(R,M)
\,,\\
\widetilde H_{\LC}^n(R,M)
&=
\ker\big(\widetilde d|_{\widetilde C_{\LC}^{n}(R,M)}\big)
/
\widetilde d\big(\widetilde C_{\LC}^{n-1}(R,M)\big)
\,.
\end{split}
\end{equation}
\end{definition}
%
Due to Lemma \ref{lem:basic-d},
we have a short exact sequence of complexes
\begin{equation}\label{eq:lca-ses}
0 \to \partial\widetilde C_{\LC}(R,M) \to \widetilde C_{\LC}(R,M) \to \widetilde C_{\LC}(R,M) / \partial\widetilde C_{\LC}(R,M) \to 0 \,,
\end{equation}
which leads to a long exact sequence of cohomology. 
By Lemma \ref{lem:Ytilde-lca}, in the case when $R$ is free as an $\mb F[\partial]$-module, we obtain the long exact sequence \cite{BKV99}:
\begin{equation}\label{eq:lca-les}
\begin{split}
0 \to H^0\big(\partial\widetilde C_{\LC}(R,M)\big)
&\to \widetilde H_{\LC}^0(R,M) \to H_{\LC}^0(R,M) \to
\\
\to H^1\big(\partial\widetilde C_{\LC}(R,M)\big) 
&\to \widetilde H_{\LC}^1(R,M) \to H_{\LC}^1(R,M) \to
\\
\to H^2\big(\partial\widetilde C_{\LC}(R,M)\big)
&\to \widetilde H_{\LC}^2(R,M) \to H_{\LC}^2(R,M) \to\cdots
\,.
\end{split}
\end{equation}
Note that, by Lemma \ref{lem:partial-inj}, we also have
$H^n\big(\partial\widetilde C_{\LC}(R,M)\big) \simeq \widetilde H_{\LC}^n(R,M)$
for all $n\ge 1$.
As a consequence, we obtain the following result.

\begin{proposition}[cf.\ \cite{BKV99}]\label{prop:Ytilde-lca}
Let\/ $R$ be an LCA, which is free as an\/ $\mb F[\partial]$-module, and let\/ $M$ be an\/ $R$-module. Suppose that\/ $\widetilde H_{\LC}^n(R,M)=0$ for all\/ $n\ge0$. Then\/
$H_{\LC}^n(R,M)=0$ for all\/ $n\ge0$.
\end{proposition}

\subsection{Lie derivatives, contractions and Cartan's formula}\label{sec:lie-der}

For $a\in R$ and $\widetilde Y\in\widetilde C_{\LC}^n(R,M)$, 
we define the \emph{Lie derivative} of $\widetilde Y$ by $a$,
as the linear map 
$$
a_\la\widetilde Y\colon R^{\otimes n}\to M[\lambda_1,\dots,\lambda_n,\la]
\,,
$$
given by the formula (cf.\ \cite[Section 5]{BKV99}):
\begin{equation}\label{eq:lca-modstr}
\begin{split}
(a_\la\widetilde Y&)_{\lambda_1,\dots,\lambda_n}
(a_1\otimes\cdots\otimes a_n) 
=
a_{\lambda}
\bigl(\widetilde Y_{\lambda_1,\dots,\lambda_n}
(a_1\otimes\cdots\otimes a_n) \bigr)
\\
& +
\sum_{i=1}^n
(-1)^{\delta_i}
\widetilde Y_{\lambda_1,\dots,\lambda+\lambda_i,\dots,\lambda_n}
(a_1\otimes\cdots\otimes [a_{\lambda}a_i]\otimes\cdots\otimes a_n) 
\,,
\end{split}
\end{equation}
where
\begin{equation}\label{eq:deltai}
\delta_i =
p(a) \bigl(\bar p(\widetilde Y)+\bar p(a_1)+\dots+\bar p(a_{i-1})+1\bigr)+1
\,.
\end{equation}
In particular,
\begin{equation}\label{eq:deltai-even}
\delta_i=1 \quad\text{if}\quad p(a)=\bar0
\,,
\end{equation}
and
\begin{equation}\label{eq:deltai-odd}
\delta_i=\bar p(\widetilde Y) \quad\text{if}\quad p(a)=p(a_1)=\dots=p(a_n)=\bar1
\,.
\end{equation}
It is easy to check that $a_\la\widetilde Y$ satisfies the sesquilinearity \eqref{eq:sesq} 
and symmetry \eqref{eq:skew};
hence, $a_\la\widetilde Y$ can be viewed as an element of $\widetilde C_{\LC}^n(R,M)[[\la]]$.
Note that the linear map $a_\lambda\colon \widetilde C_{\LC}^n(R,M)\to\widetilde C_{\LC}^n(R,M)[[\la]]$
has parity $p(a)$.

\begin{proposition}\label{pmodstr}
For\/ $a\in R$ and\/ $\widetilde Y\in\widetilde C_{\LC}^n(R,M)$, we have{\rm:}
\begin{enumerate}[(a)]
\item 
$(\partial a)_\lambda\widetilde Y=-\lambda \, a_\lambda\widetilde Y$ \; and \;
$a_\lambda(\partial\widetilde Y)=(\partial+\lambda) (a_\lambda\widetilde Y)$ \;
in \; $\widetilde C_{\LC}^n(R,M)[[\la]];$
\item
$a_\lambda(b_\mu\widetilde Y)
-(-1)^{p(a)p(b)} b_\mu(a_\lambda\widetilde Y)
=[a_\lambda b]_{\lambda+\mu}\widetilde Y$ \;
in \; $\widetilde C_{\LC}^n(R,M)[[\la,\mu]]$.
\end{enumerate}
In other words, 
formula \eqref{eq:lca-modstr} endows\/ $\widetilde C_{\LC}^n(R,M)$ 
with the structure of an $R$-module, for every\/ $n\ge0$,
if we allow formal power series in\/ $\lambda$ for the action.
\end{proposition}
\begin{proof}
This can be checked by a straightforward computation, which is left to the reader. Another proof can be obtained by using the relationship to the cohomology of the annihilation Lie algebra and the well-known action of a Lie algebra on its cohomology complex (see \cite{BKV99}).
\end{proof}

We also define the \emph{contraction} of $\widetilde Y\in\widetilde C_{\LC}^n(R,M)$ 
by $a\in R$, as the linear map 
$$
\iota_\la(a)\widetilde Y\colon R^{\otimes (n-1)} \to M[\lambda_1,\dots,\lambda_{n-1},\lambda] \,,
$$
given by
\begin{equation}\label{eq:lca-iota2}
\begin{split}
\bigl(\iota_\la(a)\widetilde Y &\bigr)_{\lambda_1,\dots,\lambda_{n-1}}
(a_1\otimes\cdots\otimes a_{n-1}) 
\\
&= (-1)^{\bar p(a)\bar p(\widetilde Y)} \,
\widetilde Y_{\la,\lambda_1,\dots,\lambda_{n-1}}
(a\otimes a_1 \otimes\cdots\otimes a_{n-1})
\,.
\end{split}
\end{equation}
As before, it is easy to check that $\iota_\la(a)\widetilde Y$ satisfies
the sesquilinearity \eqref{eq:sesq} 
and symmetry \eqref{eq:skew};
hence, we can view $\iota_\lambda(a)$ as a map
\begin{equation}\label{eq:lca-iota1}
\iota_\la(a) \colon \widetilde C_{\LC}^n(R,M) \to \widetilde C_{\LC}^{n-1}(R,M)[[\la]]
\,, \qquad a\in R\,,
\end{equation}
Note that $\iota_\lambda(a)$ is a linear map of parity $\bar p(a)$.

\begin{proposition}[cf.\ \cite{BKV99}]\label{pcartan}
On the basic complex\/ $\widetilde C_{\LC}(R,M)$, we have
Cartan's formula
\begin{equation}\label{eq:cartan-lc}
a_\la = \bigl[\iota_\la(a), \widetilde d \bigr]
:= \iota_\la(a) \, \widetilde d - (-1)^{\bar p(a)} \widetilde d \, \iota_\la(a)
\,, \qquad a\in R\,.
\end{equation}
\end{proposition}
\begin{proof}
Same as the proof of Proposition \ref{pmodstr}.
\end{proof}

\begin{corollary}\label{ccartan}
The action of the LCA $R$ on the basic complex\/ $\widetilde C_{\LC}(R,M)$
given by the Lie derivatives commutes with the differential\/ $\widetilde d$, and it induces a trivial action on its cohomology. In other words, if we write\/ $a_\la = \sum_{n\ge0} a_{(n)} \la^n/n!$ for\/ $a\in R$, then all linear operators\/ $a_{(n)}$ act as zero on\/ $\widetilde H_{\LC}(R,M)$.
\end{corollary}

By Proposition \ref{pmodstr}(a), for $a\in R$ we have $[\partial,a_\lambda]=-\lambda a_\lambda$.
Hence, 
the zero mode $a_{(0)} = a_\la |_{\la=0}\,\in\End\widetilde C_{\LC}^n(R,M)$
commutes with $\partial$.
In fact, the formula \eqref{eq:lca-modstr} for the $0$-th mode
gives a well-defined linear operator,
of parity $p(a)$, on the complex $C_{\LC}(R,M)$:
\begin{equation}\label{eq:lca-modstr0}
\begin{split}
(a_{(0)} Y&)_{\lambda_1,\dots,\lambda_n}
(a_1\otimes\cdots\otimes a_n) 
=
a_{(0)}
\bigl(Y_{\lambda_1,\dots,\lambda_n}
(a_1\otimes\cdots\otimes a_n) \bigr)
\\
& +
\sum_{i=1}^n
(-1)^{\delta_i}
Y_{\lambda_1,\dots,\lambda_n}
(a_1\otimes\cdots\otimes (a_{(0)}a_i)\otimes\cdots\otimes a_n) 
\,,
\end{split}
\end{equation}
where $\delta_i$ are as in \eqref{eq:deltai}.
By Proposition \ref{pmodstr}(b),
we also have $[a_{(0)},b_{(0)}] = (a_{(0)}b)_{(0)}$.
Hence, \eqref{eq:lca-modstr0}
gives a representation of the Lie superalgebra $R/\partial R$ on 
the complex $C_{\LC}(R,M)$.

For $a\in R$, we also have a well-defined contraction operator
$\iota_0(a) \colon C_{\LC}^n(R,M) \to C_{\LC}^{n-1}(R,M)$,
of parity $\bar p(a)$,
given by (cf.\ \eqref{eq:lca-iota2})
\begin{equation}\label{eq:lca-iota3}
\begin{split}
\bigl(\iota_0(a) Y &\bigr)_{\lambda_1,\dots,\lambda_{n-1}}
(a_1\otimes\cdots\otimes a_{n-1}) 
\\
&= (-1)^{\bar p(a)\bar p(Y)} \,
Y_{0,\lambda_1,\dots,\lambda_{n-1}}
(a\otimes a_1 \otimes\cdots\otimes a_{n-1})
\,.
\end{split}
\end{equation}

\begin{proposition}\label{pcartan2}
We have Cartan's formula on\/ $C_{\LC}(R,M)${\rm{:}}
\begin{equation}\label{eq:cartan2}
a_{(0)} = [\iota_0(a), d]
:= \iota_0(a) \, d - (-1)^{\bar p(a)} d \, \iota_0(a)
\,, \qquad a\in R\,.
\end{equation}
Consequently, the action of the Lie superalgebra\/ $R/\partial R$ on\/ $C_{\LC}(R,M)$ 
by zero modes commutes with the differential\/ $d$ and induces the
trivial action on the cohomology $H_{\LC}(R,M)$.
\end{proposition}
\begin{proof}
The proof is straightforward from the definitions. 
\end{proof}


For the rest of this section, we discuss the cohomology of the Virasoro and the affine LCA's.


\subsection{Cohomology of the Virasoro LCA}\label{sec:2.5}
In this subsection, we will compute the cohomology of the Virasoro LCA $R^\vir$ from Example \ref{ex:virasoro-lca}.
First, consider the Virasoro LCA at central charge zero, namely $\bar R^\vir = R^{\vir}/\mb F C = \mb F[\partial]L$ with the $\lambda$-bracket $[L_\lambda L]=(\partial+2\lambda)L$. Its cohomology was computed in \cite{BKV99}.
The coefficients will be taken in the trivial $\bar R^\vir$-module $\mb F$ (where both $L$ and $\partial$ act by zero) or in the modules $M_{\Delta}$ for $\Delta\in\mb F$. The latter are free of rank one over $\mb F[\partial]$ and are given by:
\begin{equation}\label{mdelta}
M_{\Delta} = \mb F[\partial] v \,, \quad L_\lambda v = (\partial+\Delta\lambda)v\,, \qquad \Delta\in\mb F\,.
\end{equation}

\begin{proposition}[\cite{BKV99}]\label{pbarvir}
We have
\begin{align}
\label{barvir1}
\dim H^n_{\LC}(\bar R^\vir,\mb F) &= \begin{cases} 1, \quad\text{for } \; n=0,2,3, \\ 0, \quad \text{otherwise}, \end{cases}
\\[12pt]
\label{barvir2}
\dim H^n_{\LC}(\bar R^\vir,M_{1-(3r^2\pm r)/2}) &= \begin{cases} 
2, \quad\text{for } \; n=r+1, \\ 
1, \quad\text{for } \; n=r,r+2, \\ 
0, \quad \text{otherwise}, \end{cases}
\end{align}
and\/ $H^n_{\LC}(\bar R^\vir,M_{\Delta}) = 0$ if\/ $\Delta\ne 1-(3r^2\pm r)/2$ for any\/ $r\in\mb Z$.
\end{proposition}

The cohomology of the LCA $R^\vir$ can be obtained from Propositions \ref{prop:central-ext} and \ref{pbarvir}.

\begin{theorem}\label{tcohvir}
We have
\begin{align}
\label{barvir3}
\dim H^n_{\LC}(R^\vir,\mb F) &= \begin{cases} 1, \quad\text{for } \; n=0,3, \\ 0, \quad \text{otherwise}, \end{cases}
\\[12pt]
\label{barvir4}
\dim H^n_{\LC}(R^\vir,M_{1-(3r^2\pm r)/2}) &= \begin{cases} 
2, \quad\text{for } \; n=r+1, \\ 
1, \quad\text{for } \; n=r,r+2, \\ 
0, \quad \text{otherwise}, \end{cases}
\end{align}
and\/ $H^n_{\LC}(R^\vir,M_{\Delta}) = 0$ if\/ $\Delta\ne 1-(3r^2\pm r)/2$ for any\/ $r\in\mb Z$.
\end{theorem}
\begin{proof}
Due to Proposition \ref{prop:central-ext} and Corollary \ref{cor:central-ext}, we only need to consider the cohomology $H^n_{\LC}(R^\vir,\mb F)$ for $n=1,2$. For the $R^\vir$-module $M=\mb F$, we have $U=\mb F$ and
\begin{equation}\label{barvir5}
\begin{split}
\bigl( &\dim H_{\LC}^1(R^\vir,\mb F)-\dim H_{\LC}^1(\bar{R}^\vir,\mb F) \bigr)  \\
&+\bigl( \dim H_{\LC}^2(\bar{R}^\vir,\mb F) - \dim H_{\LC}^2(R^\vir,\mb F) \bigr)
= 1
\end{split}
\end{equation}
(see \eqref{eq:H01}). By Theorem \ref{thm:lowcoho}(c) and Proposition \ref{pbarvir}, we have $H_{\LC}^2(\bar{R}^\vir,\mb F) = \mb F Y$, where 
$$
Y_{\la_1,\la_2}(L \otimes L) = \frac{\la_1^3}{12} + \langle \la_1+\la_2\rangle
$$
is the $2$-cocycle giving the central extension $R^\vir$. However, the image of $Y$ in $H_{\LC}^2(R^\vir,\mb F)$ under the surjective map \eqref{eq:h-maps} is trivial, because $Y=dZ$ for the $1$-cochain $Z$ on $R^\vir$ defined by 
$$
Z_{\la}(L)=\langle \la\rangle \,, \qquad Z_{\la}(C)=1+\langle \la\rangle \,. 
$$
Therefore, $H_{\LC}^2(R^\vir,\mb F) =0$. Then \eqref{barvir5} gives $H_{\LC}^1(R^\vir,\mb F) = H_{\LC}^1(\bar{R}^\vir,\mb F) = 0$.
\end{proof}

\begin{remark}\label{rem:label}
By the proof of Theorem 7.1 in \cite{BKV99}, the nontrivial $3$-cocycle $Y\in H_{\LC}^3(R^\vir,\mb F)$ is given explicitly by
$$
Y_{\la_1,\la_2,\la_3}(L \otimes L \otimes L) = (\la_1-\la_2)(\la_1-\la_3)(\la_2-\la_3) + \langle \la_1+\la_2+\la_3\rangle \,.
$$
\end{remark}
As a special case of Theorem \ref{tcohvir}, 
since the adjoint representation of $\bar R^\vir$ is $M_2$, we deduce that its cohomology is trivial:
\begin{equation}\label{cvirmdelta}
H^n_{\LC}(R^\vir,\bar R^\vir) = H^n_{\LC}(\bar R^\vir,\bar R^\vir) = 0 \,, \quad \text{for all} \quad n\ge0 \,.
\end{equation}

\subsection{Cohomology of the affine LCA}\label{sec:2.6}
Throughout this subsection, $\mf g$ will be a finite-dimensional simple Lie algebra.
We will compute the cohomology of the affine LCA $\cur\mf g$ from Example \ref{ex:affine-lca}.
Denote by $\overline{\cur}\,\mf g = \mb F[\partial]\mf g$ the affine LCA at level $0$, with the $\la$-bracket $[a_\la b]=[a,b]$ for
$a,b\in\mf g$.

First, consider the trivial module $\mb F$, where $\partial$ also acts by zero. 
It is well known (see e.g.\ \cite{C55})
that the Lie algebra cohomology of $\mf g$ with coefficients in $\mb F$ is given by the $\mf g$-invariants in the exterior algebra:
\begin{equation}\label{cohg}
H^\bullet(\mf g,\mb F) = \bigl(\textstyle\bigwedge\nolimits^\bullet \mf g^* \bigr)^{\mf g} \,,
\end{equation}
and is generated as an algebra by homogeneous elements of degrees $2m_i+1$ ($i=1,\dots,\rank\mf g$), where $m_i$ are the exponents of $\mf g$.

\begin{proposition}[\cite{BKV99}]\label{pbarcurf}
If\/ $\mf g$ is a finite-dimensional simple Lie algebra,
we have 
$$H^n_{\LC}(\overline{\cur}\,\mf g,\mb F) \simeq H^n(\mf g,\mb F) \oplus H^{n+1}(\mf g,\mb F)
\,, \qquad n\ge0 \,.
$$
Explicitly, under this isomorphism, a Lie algebra\/ $n$-cocycle\/ 
$\al\in H^n(\mf g, \mb F)$ 
corresponds to the\/ $n$-cocycle\/ 
$Y \in C_{\LC}^{n}(\overline{\cur}\,\mf g,\mb F)$
defined by
$$
Y_{\la_1,\dots,\la_n}(u) 
= \al(u) + \langle \la_1+\dots+\la_n \rangle 
\,, \qquad u\in\mf g^{\otimes n} \,.
$$
A Lie algebra\/ $(n+1)$-cocycle\/ $\varphi\in H^{n+1}(\mf g, \mb F)$ 
corresponds to the unique\/ $n$-cocycle\/ 
$Y =\pi\circ\widetilde Y \in C_{\LC}^{n}(\overline{\cur}\,\mf g,\mb F)$,
where\/ 
$\widetilde Y \in \widetilde C_{\LC}^{n}(\overline{\cur}\,\mf g,\mb F)$
has the form
$$
\widetilde Y_{\la_1,\dots,\la_n}(u) 
= \sum_{j=1}^n f_j(u) \, \la_j
\,, \qquad u\in\mf g^{\otimes n} \,, \quad f_j\in (\mf g^{\otimes n})^* \,,
$$
and satisfies
$$
(\widetilde d \widetilde Y)_{\la_0,\dots,\la_n}(v)
= (\la_0+\dots+\la_n) \varphi(v)
\,, \qquad v\in\mf g^{\otimes (n+1)} \,.
$$
\end{proposition}
\begin{proof}
This follows from \cite[Theorem 8.1]{BKV99} and its proof.
\end{proof}

Using Proposition \ref{prop:central-ext}, we can find the cohomology of $\cur\mf g$ with trivial coefficients.

\begin{theorem}\label{tcohcurgf}
If\/ $\mf g$ is a finite-dimensional simple Lie algebra,
then
\begin{align}
\label{curgf1}
\dim H^0_{\LC}(\cur\mf g,\mb F) &= 1 \,, \qquad
H^1_{\LC}(\cur\mf g,\mb F) = H^2_{\LC}(\cur\mf g,\mb F) = 0 
\intertext{and}
\label{curgf2}
H^n_{\LC}(\cur\mf g,\mb F) &= H^n_{\LC}(\overline{\cur}\,\mf g,\mb F) 
\simeq H^n(\mf g,\mb F) \oplus H^{n+1}(\mf g,\mb F)\,, \qquad n\ge 3\,.
\end{align}
\end{theorem}
\begin{proof}
The proof is the same as that of Theorem \ref{tcohvir}.
\end{proof}

For any $\mf g$-module $V$, we have the $\overline{\cur}\,\mf g$-module $M_V$ defined by
\begin{equation}\label{curgmv}
M_V = \mb F[\partial] V \,, \quad a_\la v = av \,, \qquad a\in\mf g\,, \; v\in V \,.
\end{equation}
In the case when $V$ is an irreducible $\mf g$-module, the cohomology of $\overline{\cur}\,\mf g$ with coefficients in $M_V$ was computed in \cite[Section 8.2]{BKV99}. In order to recall the result, we first need to remind some notation and results on affine Lie algebras (see \cite{K90}).

Let us denote by $\mf h$, $W$, $\rho$, $\theta$, $h^\vee$, respectively, a fixed Cartan subalgebra, the Weyl group, the half-sum of positive roots, the highest root, and the dual Coxeter number of $\mf g$. Let $\hat{\mf g} = \mf g[t,t^{-1}] \oplus \mb F K \oplus \mb F d$ be the affine Kac--Moody algebra associated to $\mf g$. 
Let $\hat{\mf h} = \mf h \oplus\mb F K \oplus\mb F d$, and denote by $\bar\la\in\mf h^*$ the restriction of $\la\in\hat{\mf h}^*$ to $\mf h$. 
Recall that the simple roots of $\hat{\mf g}$ are 
$\hat\al_0=\delta-\theta$ and $\hat\al_i=\al_i$ $(i=1,\dots,\rank\mf g)$, where $\al_i$ are the simple roots of $\mf g$ and $\delta$ is the null root of $\hat{\mf g}$ (which corresponds to the central element $K$ under the isomorphism $\hat{\mf h}^* \simeq \hat{\mf h}$).
Then $\hat\rho\in\hat{\mf h}^*$ is defined by the property that $\langle\hat\rho,\hat\al_i\rangle = 1$ for all $i$. We can take $\hat\rho=\rho+h^\vee\La_0$, where $\La_0$ is the $0$-th fundamental weight, defined by $\langle\La_0,\hat\al_i\rangle = \delta_{i,0}$ for all $i$.
The affine Weyl group $\hat W$ is the semidirect product of $W$ and the group of translations $t_\gamma$, where $\gamma$ is in the $\mb Z$-span of the long roots of $\mf g$. Recall that 
$\overline{t_\gamma \la} = \bar\la+\langle \la, K \rangle \ga$ 
for $\la\in\hat{\mf h}^*$. Finally, for $\hat w\in \hat W$, we denote by $\ell(\hat w)$ the length of $\hat w$. 

For $\La\in\mf h^*$, we define $\ell(\La)$ as $\ell(\hat w)$ if we can write $\La = \overline{\hat w\hat\rho - \hat\rho}$ for some $\hat w\in \hat W$, and $\ell(\La)=+\infty$ if such $\hat w$ does not exist. Note that
\begin{equation}\label{curgmv2}
\overline{\hat w\hat\rho - \hat\rho} = w\rho-\rho + h^\vee \gamma \,, \qquad\text{for}\quad \hat w = t_\gamma w \,, \; w\in W \,.
\end{equation}
An important special case is when $\La=\theta$, the highest weight of the adjoint representation. Then the simple reflection $r_{\hat\al_0} \in \hat W$ satisfies $r_{\hat\al_0} \hat\rho = \hat\rho-\hat\al_0$. Hence, $\theta=\overline{r_{\hat\al_0}\hat\rho - \hat\rho}$ and 
\begin{equation}\label{curgmv3}
\ell(\theta)=1 \,.
\end{equation}
For an irreducible $\mf g$-module $V$, we let $\ell(V)=\ell(\La)$ if $V$ is finite dimensional with highest weight $\La$.
When $V$ is an infinite-dimensional irreducible $\mf g$-module, we let $\ell(V)=+\infty$.

\begin{proposition}[\cite{BKV99}]\label{pbarcurmv}
With the above notation, for any irreducible\/ $\mf g$-mod\-ule\/ $V$, we have
$$
H^n_{\LC}(\overline{\cur}\,\mf g, M_V) \simeq H^{n-\ell(V^*)} (\mf g,\mb F) \,, \qquad n\ge0 \,,
$$
where\/
$V^*$ is the contragredient\/ $\mf g$-module,
and we let\/ $H^n=0$ for all\/ $n<0$, including\/ $n=-\infty$.
\end{proposition}

\begin{proposition}\label{prop:barcurg}
We have
\begin{equation}\label{curgmv6}
H^n_{\LC}(\overline{\cur}\,\mf g, \overline{\cur}\,\mf g) \simeq H^{n-1} (\mf g,\mb F)
\,, \qquad n\ge0 \,,
\end{equation}
where we let\/ $H^n=0$ for\/ $n<0$.
In particular, all derivations of\/ $\overline{\cur}\,\mf g$ are inner
and all first-order deformations are trivial.

Explicitly, 
for a Lie algebra\/ $(n-1)$-cocycle\/ $\be\in\bigl( \bigwedge\nolimits^{n-1} \mf g^* \bigr)^{\mf g}$,
the corresponding\/ $n$-cocycle\/ $Y \in C_{\LC}^{n}(\overline{\cur}\,\mf g, \overline{\cur}\,\mf g)$
is given by
\begin{equation}\label{curgmv7}
Y_{\la_1,\dots,\la_{n}}(a_1 \otimes\dots\otimes a_{n}) 
= \sum_{i=1}^{n} (-1)^{i+1} 
\be(a_1\wedge\stackrel{i}{\check{\dots}}\wedge a_{n}) \la_i a_i 
\,, \qquad a_i\in\mf g \,.
\end{equation}
\end{proposition}
\begin{proof}
The adjoint module is $\overline{\cur}\,\mf g = M_{\mf g}$. 
For $V=\mf g$, we have $V^*\simeq V$ and $\La=\theta$.
Then \eqref{curgmv6} follows from Proposition \ref{pbarcurmv} and \eqref{curgmv3}.
One can check that the isomorphism is given by \eqref{curgmv7} 
by a direct calculation. We will give a more conceptual proof later, by using variational PVA cohomology
(see Remark \ref{rem:barcurg} below).
\end{proof}

The following result is an immediate consequence of Corollary \ref{cor:central-ext}
and Proposition \ref{pbarcurmv}.
\begin{theorem}\label{tbarcurmv}
With the above notation, for any 
finite-dimensional simple Lie algebra\/ $\mf g$ and
an irreducible\/ $\mf g$-mod\-ule\/ $V$, we have
\begin{equation}\label{curgmv4}
H^n_{\LC}(\cur\mf g, M_V) \simeq H^{n-\ell(V^*)} (\mf g,\mb F) \,, \qquad n\ge0 \,,
\end{equation}
where we let\/ $H^n=0$ for all\/ $n<0$, including\/ $n=-\infty$. In particular,
\begin{equation}\label{curgmv5}
H^n_{\LC}(\cur\mf g, \overline{\cur}\,\mf g) \simeq H^{n-1} (\mf g,\mb F) \,, \qquad n\ge0 \,.
\end{equation}
\end{theorem}

\section{Variational PVA cohomology}
\label{sec:3}

In this section, we review the definitions of Poisson vertex algebra (PVA), its modules 
and the corresponding variational PVA cohomology. 
We establish a relationship between LCA and variational PVA cohomology. We prove a theorem that the Virasoro conformal weight in cohomology can be only $0$ or $1$, which is used in the next Section \ref{sec:com-pva} to compute the cohomology of our main examples.

\subsection{Poisson vertex algebras}
\label{sec:3.1}
We start by recalling the basic definitions.

\begin{definition}\label{def:pva}
Let $\mc V$ be a commutative associative unital differential superalgebra with parity $p$,
with an even derivation $\partial$.
A \emph{Poisson vertex superalgebra} (PVA) structure on $\mc V$
is an LCA $\lambda$-bracket
$\mc V\otimes \mc V\to \mc V[\lambda]$, $a\otimes b\mapsto[a_\lambda b]$,
such that the following left Leibniz rule holds ($a,b,c\in R$):
\begin{enumerate}[L1]
\setcounter{enumi}{3}
\item
$[a_\lambda bc]=[a_\lambda b] c+(-1)^{p(b)p(c)}[a_\lambda c] b$.
\end{enumerate}

By the skewsymmetry L2, this axiom is equivalent to the right Leibniz rule
\begin{enumerate}[L1']
\setcounter{enumi}{3}
\item
$[ab_\lambda c]=
(e^{\partial\partial_\lambda}a)[b_\lambda c] 
+(-1)^{p(a)p(b)}(e^{\partial\partial_\lambda}b)[a_\lambda c]$.
\end{enumerate}

A \emph{module} $M$ over the PVA $\mc V$
is a vector superspace endowed 
with a structure of a module over the differential algebra $\mc V$,
denoted by $a\otimes m\mapsto a m$,
and with a structure of a module over the LCA $\mc V$,
denoted by $a\otimes m\mapsto a_\lambda m$,
satisfying 
\begin{enumerate}[M1]
\setcounter{enumi}{2}
\item
$a_\lambda(bm)=[a_\lambda b]m+(-1)^{p(a)p(b)}b(a_\lambda m)$;
\end{enumerate}
\begin{enumerate}[M1']
\setcounter{enumi}{2}
\item
$(ab)_\lambda m=
(e^{\partial\partial_\lambda}a)(b_\lambda m) 
+(-1)^{p(a)p(b)}(e^{\partial\partial_\lambda}b)(a_\lambda m)$.
\end{enumerate}
\end{definition}
A PVA $\mc V$ is called \emph{graded}
if there is a grading by $\mb F[\partial]$-submodules
$$
\mc V=\bigoplus_{n\in\mb Z_+}\mc V[n]
\,,
$$
such that ($m,n\in\mb Z_+$)
\begin{equation}\label{eq:grading}
\mc V[m]\mc V[n]\subset\mc V[m+n]
\,\,,\qquad
[\mc V[m]_\lambda\mc V[n]]
\subset
(\mc V[m+n-1])[\lambda]
\,.
\end{equation}
If $\mc V$ is a graded PVA,
a $\mc V$-module $M$ is \emph{graded} if there is a grading by $\mb F[\partial]$-submodules
$$
M=\bigoplus_{n\in\mb Z_+}M[n]\,,
$$
such that ($m,n\in\mb Z_+$)
\begin{equation}\label{eq:graded-module}
\mc V[m]M[n]\subset M[m+n]
\,,\qquad
\mc V[m]_\lambda M[n]
\subset
(M[m+n-1])[\lambda]
\,.
\end{equation}

Notice that every PVA is a module over itself, called the adjoint module.

\subsection{Universal PVA over an LCA}
\label{sec:3.25}

Given an LCA $R$,
there is the canonical universal PVA $\mc V(R)$ over $R$ constructed as follows.
As a commutative associative superalgebra it is $\mc V(R)=S(R)$,
the symmetric superalgebra over $R$, viewed as a vector superspace.
The endomorphism $\partial\in\End R$ uniquely extends 
to an even derivation of the superalgebra $\mc V(R)$.
Moreover, the $\lambda$-bracket on $R$
extends uniquely to a PVA $\lambda$-bracket on $\mc V(R)$
by the Leibniz rules L4 and L4'.
Note that the universal PVA $\mc V(R)$ over the LCA $R$
is automatically graded, by the usual symmetric superalgebra degree.

If $C\in R$ is such that $\partial C=0$ and $c\in\mb F$, 
then $C$ is central and $\mc V(R)(C-c)\subset \mc V(R)$ is a PVA ideal,
so we can consider the quotient PVA
$$
\mc V^c(R)=\mc V(R)/\mc V(R)(C-c)
\,.
$$
Using the above constructions, we obtain,
starting from Examples \ref{ex:boson-lca}--\ref{ex:virasoro-lca},
the corresponding PVA's.
\begin{example}[Free superboson PVA]\label{ex:boson-pva}
Let $\mf h$ be a finite-dimensional superspace, with parity $p$, and a supersymmetric nondegenerate bilinear form $(\cdot|\cdot)$, as in Example \ref{ex:boson-lca}.
The universal PVA over the free superboson LCA $R^b_{\mf h}$
is the symmetric superalgebra
$$
\mc V(R^b_{\mf h})
=
S\bigl(\mb F[\partial]\mf h \oplus \mb F K\bigr)
\,,
$$
endowed with the $\lambda$-bracket defined on generators by \eqref{eq:boson}
and extended uniquely to $\mc V(R^b_{\mf h})$ by the left and right Leibniz rules L4 and L4'
and the sesquilinearity conditions L1.
This is a graded PVA by the usual polynomial degree,
where $\deg(\partial^n a)=\deg K=1$ for $a\in\mf h$.

The \emph{free superboson} PVA is the
quotient of $\mc V(R^b_{\mf h})$ 
by the ideal $\mc V(R^b_{\mf h})(K-1)$:
$$
\mc B_{\mf h}
=
\mc V^1(R^b_{\mf h})
= S\bigl(\mb F[\partial]\mf h\bigr)
\,,
$$
with the $\lambda$-bracket as in \eqref{eq:boson} with $K=1$.
Note that, since the relation $K-1$ is not homogeneous,
the free superboson PVA $\mc B_{\mf h}$ is not graded
(though $\mc V(R^b_{\mf h})$ is).

When $\mf h$ is purely even, $\mc B_{\mf h}$ is called the \emph{free boson} PVA.
In that case, it is isomorphic, as a differential algebra, to the algebra of differential polynomials in $N$ generators:
$$
\mc B_{\mf h}
=
\mb F\bigl[u_i^{(n)} \,\big|\, i=1,\dots,N,\,n\in\mb Z_+\bigr]
\,, \qquad u_i^{(n)} = \partial^n u_i \,,
$$
where $\{u_1,\dots,u_N\}$ is an $\mb F$-basis of $\mf h$.
\end{example}
\begin{example}[Free superfermion PVA]\label{ex:fermion-pva}
Let $\mf h$ be a finite-dimensional superspace, with parity $p$, and a super-skewsymmetric nondegenerate bilinear form $(\cdot|\cdot)$, as in Example \ref{ex:fermion-lca}.
The universal PVA over the free superfermion LCA $R^f_{\mf h}$
is the symmetric superalgebra
$$
\mc V(R^f_{\mf h})
=
S\bigl(\mb F[\partial]\mf h \oplus \mb F K\bigr)
\,,
$$
with the $\lambda$-bracket defined on generators by \eqref{eq:fermion}
and extended uniquely to $\mc V(R^f_{\mf h})$ by the left and right Leibniz rules L4 and L4'
and the sesquilinearity conditions L1.
This is a graded PVA by the usual polynomial degree,
where $\deg(\partial^n a)=\deg K=1$ for $a\in\mf h$.

The \emph{free superfermion} PVA is the
quotient of $\mc V(R^f_{\mf h})$ 
by the ideal $\mc V(R^f_{\mf h})(K-1)$:
$$
\mc F_{\mf h}
=
\mc V^1(R^f_{\mf h})
= S\bigl(\mb F[\partial]\mf h\bigr)
\,,
$$
with the $\lambda$-bracket as in \eqref{eq:fermion} with $K=1$.
When $\mf h$ is purely odd, $\mc F_{\mf h}$ is called just the \emph{free fermion} PVA.
In that case, it is isomorphic, as a differential algebra, 
to the algebra of differential polynomials in $N$ odd generators:
$$
\mc F_{\mf h}
=
\textstyle\bigwedge\bigl(u_i^{(n)} \,\big|\, i=1,\dots,N,\,n\in\mb Z_+ \bigr)
\,, \qquad u_i^{(n)} = \partial^n u_i
\,,
$$
where $\{u_1,\dots,u_N\}$ is an $\mb F$-basis of $\mf h$.
\end{example}
\begin{example}[Affine PVA]\label{ex:affine-pva}
As in Example \ref{ex:affine-lca},
let $\mf g$ be a Lie algebra with a nondegenerate symmetric invariant bilinear form $(\cdot\,|\,\cdot)$.
The universal PVA over the affine LCA $\cur\mf g$
is the (purely even) symmetric algebra
$$
\mc V(\cur\mf g)
=
S\bigl(\mb F[\partial]\mf g \oplus \mb F K\bigr)
\,,
$$
endowed with the $\lambda$-bracket defined on generators by \eqref{eq:current}
and extended uniquely
to $\mc V(\cur\mf g)$ by the left and right Leibniz rules L4 and L4'
and the sesquilinearity conditions L1.
This is a graded PVA by the usual polynomial degree,
where $\deg(\partial^n a)=\deg K=1$ for $a\in\mf g$.

The \emph{affine} PVA at \emph{level} $k\in\mb F$ is defined as the quotient of $\mc V(\cur\mf g)$ 
by the ideal $\mc V(\cur\mf g)(K-k)$:
$$
\mc V^k_{\mf g}
=
\mc V^k(\cur\mf g)
=
S\bigl(\mb F[\partial]\mf g \bigr)
\,,
$$
with the $\lambda$-bracket as in \eqref{eq:current} with $K=k$.
As a differential algebra, it is isomorphic to the algebra of differential polynomials
$$
\mc V^k_{\mf g}
=
\mb F\bigl[u_i^{(n)}
\,\big|\, i=1,\dots,N,\,n\in\mb Z_+\bigr] 
\,, \qquad u_i^{(n)} = \partial^n u_i
\,,
$$
where $\{u_1,\dots,u_N\}$ is an $\mb F$-basis of $\mf g$.
\end{example}
\begin{example}[Virasoro PVA]\label{ex:virasoro-pva}
The universal PVA over the Virasoro LCA $R^{\Vir}$ from Example \ref{ex:virasoro-lca}
is the (purely even) algebra of polynomials
$$
\mc V(R^{\Vir})
=
\mb F\bigl[C,L^{(n)}
\,\big|\, n\in\mb Z_+\bigr]
\,,
$$
with the even derivation $\partial$ given by
$$
L^{(n)} = \partial^n L
\,,\qquad
\partial C=0
\,,
$$
endowed with the $\lambda$-bracket defined on generators by \eqref{eq:vir}
and extended uniquely
to $\mc V(R^{\Vir})$ by the left and right Leibniz rules L4 and L4'
and the sesquilinearity conditions L1.
This is a graded PVA by the usual polynomial degree,
where $\deg L^{(n)}=\deg C=1$.

The \emph{Virasoro} PVA of \emph{central charge} $c\in\mb F$
is the quotient of $\mc V(R^{\Vir})$ 
by the ideal $\mc V(R^{\Vir})(C-c)$:
$$
\PVir^c
=
\mc V^c(R^{\Vir})
=
\mb F\bigl[L^{(n)}
\,\big|\, n\in\mb Z_+\bigr] 
\,, \qquad L^{(n)} = \partial^n L
\,,
$$
with the $\lambda$-bracket as in \eqref{eq:vir} with $C=c$.
\end{example}
\begin{proposition}
\label{prop:rep-sr}
Let\/ $R$ be an LCA and consider the universal PVA\/ $\mc V(R)$.
Let\/ $M$ be an\/ $\mb F[\partial]$-module.
\begin{enumerate}[(a)]
\item
A structure of a PVA\/ $\mc V(R)$-module on\/ $M$
is the same as 
a structure of an LCA\/ $R$-module on\/ $M$,
$R\otimes M\to M[\lambda]$, $a\otimes m\mapsto a_\lambda m$,
together with an\/ $\mb F[\partial]$-module homomorphism\/
$R\otimes M\to M$, $a\otimes m\mapsto am$,
such that\/ $(a,b\in R$, $m\in M){:}$
\begin{equation}\label{eq:symm-action}
a(bm)=(-1)^{p(a)p(b)}b(am)
\,,
\end{equation}
satisfying the compatibility condition
given by the left Leibniz rule M3 $(a,b\in R$, $m\in M){:}$
\begin{equation}\label{eq:M3}
a_\lambda(bm)=[a_\lambda b]m+(-1)^{p(a)p(b)}b(a_\lambda m)
\,.
\end{equation}
\item
Let\/ $C\in R$ be such that\/ $\partial C=0$ and let\/ $c\in\mb F$.
A structure of a PVA\/ $\mc V^c(R)$-module on\/ $M$
is the same as 
a structure of an LCA\/ $R$-module on\/ $M$,
$R\otimes M\to M[\lambda]$, $a\otimes m\mapsto a_\lambda m$,
together with an\/ $\mb F[\partial]$-module homomorphism\/ 
$R\otimes M\to M$, $a\otimes m\mapsto am$,
satisfying conditions \eqref{eq:symm-action} and \eqref{eq:M3},
and such that\/ $Cm=cm$ for every\/ $m\in M$.
\end{enumerate}
\end{proposition}
\begin{proof}
The proof is straightforward. It is omitted since we will not use the statement in the rest of the paper.
\end{proof}

\subsection{Variational PVA cohomology}
\label{sec:3.2}

As in Section \ref{sec:2.2},
for a vector superspace with parity $p$, we denote by $\bar p=1-p$ the opposite parity.
Given a module $M$ over the PVA $\mc V$,
the corresponding cohomology complex 
$(C_{\PV}(\mc V,M),d)$ is defined as follows.
We let 
\begin{equation}\label{eq:W-das}
C_{\PV}(\mc V,M)=\bigoplus_{n\geq0}C_{\PV}^n(\mc V,M)
\,,
\end{equation}
where $C_{\PV}^n(\mc V,M)\subset C_{\LC}^n(\mc V,M)$ is the subspace
of cochains $Y$ satisfying the Leibniz rules:
\begin{align}
\notag
& 
Y_{\lambda_1,\dots,\lambda_n}
(a_1\otimes\cdots\otimes b_ic_i \otimes\cdots\otimes a_n) \\
\label{eq:leib}
& =
(-1)^{p(b_i)(\bar p(Y)+\bar p(a_1)+\dots+\bar p(a_{i-1}))}
(e^{\partial\partial_{\lambda_i}}b_i)
Y_{\lambda_1,\dots,\lambda_n}
(a_1\otimes\cdots\otimes c_i \otimes\cdots\otimes a_n) \\
\notag
& +
(-1)^{p(c_i)(p(b_i)+\bar p(Y)+\bar p(a_1)+\dots+\bar p(a_{i-1}))}
(e^{\partial\partial_{\lambda_i}}c_i)
Y_{\lambda_1,\dots,\lambda_n}
(a_1\otimes\cdots\otimes b_i \otimes\cdots\otimes a_n)
\,,
\end{align}
for all $i=1,\dots,n$ and $a_j,b_i,c_i \in\mc V$.

For example (cf.\ \eqref{eq:W-ksmall}):
\begin{equation}\label{eq:W-psmall}
C_{\PV}^0(\mc V,M)=M/\partial M
\,\,,\,\,\,\,
C_{\PV}^1(\mc V,M)=\Der^\partial(\mc V,M)
\,,
\end{equation}
where the second space is the space 
of linear maps $Y\colon \mc V\to M$, commuting with $\partial$
and satisfying the Leibniz rule
\begin{equation}\label{eq:deriv}
Y(ab)=(-1)^{p(a)\bar p(Y)}a \, Y(b)+(-1)^{p(b)(p(a)+\bar p(Y))}b \, Y(a)
\,.
\end{equation}
Furthermore, 
$C_{\PV}^2(\mc V,M)$ can be identified with the space of $\lambda$-brackets
$Y\colon\mc V\otimes \mc V\to M[\lambda]$
satisfying the sesquilinearity conditions L1,
symmetry with respect to the opposite parity $\bar p=1-p$ (cf.\ L2),
and the right Leibniz rule L4'.

\begin{lemma}\label{lem:id}
\begin{enumerate}[(a)]
\item
Let\/ $R$ be a subset of a PVA\/ $\mc V$, which generates it as a differential algebra. Then any\/ $n$-cochain\/ $Y\in C_{\PV}^n(\mc V,M)$ is uniquely determined by its restriction to\/ $R^{\otimes n}$.

\item
Any\/ $n$-cochain\/ $Y\in C_{\PV}^n(\mc V,M)$, with\/ $n\ge1$, vanishes whenever one of its arguments is the unit\/ $1\in\mc V$.
\end{enumerate}
\end{lemma}
\begin{proof}
Part (a) follows immediately from the Leibniz rule \eqref{eq:leib} and the sesquilinearity conditions \eqref{eq:sesq}.
For part (b), the case $n\ge2$ follows from Lemma \ref{lem:tor} since $\partial 1=0$. For $n=1$, we plug $a=b=1$ in \eqref{eq:deriv} and obtain $Y(1)=0$.
\end{proof}

\begin{proposition}[{\cite{DSK13}}]\label{prop:PVAd}
The differential\/ $d$ in equation \eqref{eq:lca-d}
preserves the subspace\/ $C_{\PV}(\mc V,M)\subset C_{\LC}(\mc V,M)$,
which then becomes a cohomology complex.
\end{proposition}
\begin{definition}\label{def:pva-coho}
Given a module $M$ over the PVA $\mc V$,
the cohomology of the complex $(C_{\PV}(\mc V,M),d)$
is called the \emph{variational PVA cohomology} of $\mc V$ with coefficients in $M$:
\begin{equation}\label{eq:h-pva}
\begin{split}
H_{\PV}(\mc V,M)
&=
\bigoplus_{n\geq0}
H_{\PV}^n(\mc V,M)
\,,\\
H_{\PV}^n(\mc V,M)
&=
\ker\big(d|_{C_{\PV}^{n}(\mc V,M)}\big)
/
d\big(C_{\PV}^{n-1}(\mc V,M)\big)
\,.
\end{split}
\end{equation}

\end{definition}
\begin{remark}\label{rem:variational-classical}
There are three closely related types of cohomology attached to a PVA $\mc V$.
The first is the \emph{variational PVA cohomology} of Definition \ref{def:pva-coho},
which is defined in \cite{DSK13} under the name of PVA cohomology.
The second is the \emph{variational Poisson cohomology}, 
defined in \cite{DSK13} when $\mc V$ is an algebra of differential functions, 
which coincides with the variational PVA cohomology of Definition \ref{def:pva-coho} 
when $\mc V$ is an algebra of differential polynomials.
Finally, the third is
the \emph{classical Poisson cohomology}, defined in \cite{BDSHK18};
it appears naturally as a classical limit of the chiral cohomology of vertex algebras.
As shown in \cite{BDSHKV19}, all three cohomology theories are isomorphic
when $\mc V$ is an algebra of differential polynomials.
\end{remark}

\subsection{Low degree cohomology}
\label{sec:3.3}

Let $\mc V$ be a PVA and $M$ be a $\mc V$-module.
As in Section \ref{sec:2.3}, 
a \emph{Casimir element} is an element $\int m\in M/\partial M$
such that  $\mc V_{-\partial}m=0$.
Denote by $\Cas(\mc V,M)\subset M/\partial M$ the space of Casimir elements.

A \emph{derivation} from the PVA $\mc V$ to the $\mc V$-module $M$
is an $\mb F[\partial]$-module homomorphism $D\colon R\to M$
satisfying the Leibniz rule \eqref{eq:deriv},
which is also a derivation from the LCA $\mc V$ to $M$,
i.e., it satisfies \eqref{eq:der}.
We say that a derivation is \emph{inner} if it has the form \eqref{eq:inner}.
Denote by $\Der(\mc V,M)$ the space of derivations from $\mc V$ to $M$,
and by $\Inder(\mc V,M)$ the subspace of inner derivations.
Note that $D\in \Der(\mc V) = \Der(\mc V,\mc V)$ if and only if $D$ is a derivation of both the product and the $\la$-bracket of $\mc V$, commuting with $\partial$. The inner derivations of $\mc V$ are those of the form $a_{(0)} = [a_\la \,\cdot\,] |_{\la=0}$.

\begin{remark}\label{rem:der}
Writing 
\begin{equation}\label{eq:n-prod}
[a_\la b] = \sum_{n\ge0} \frac{\la^n}{n!} a_{(n)} b \,, \qquad a,b\in\mc V \,, 
\end{equation}
we see from the Leibniz rule that the linear operators $a_{(n)}$ are derivations of the product of $\mc V$, for every $n\ge0$.
\end{remark}

The following result is an exact analogue of Theorem \ref{thm:lowcoho}.

\begin{theorem}[{\cite{DSK13}}]\label{thm:lowcoho2}
Let\/ $\mc V$ be a PVA and $M$ be a\/ $\mc V$-module. Then:
\begin{enumerate}[(a)]
\item
$H^0_{\PV}(\mc V,M)=\Cas(\mc V,M)$.
\item
$H^1_{\PV}(\mc V,M)=\Der(\mc V,M)/\Inder(\mc V,M)$.
\item
$H^2_{\PV}(\mc V,M)$ is the space of isomorphism classes of\/ $\mb F[\partial]$-split extensions
of the PVA\/ $\mc V$ by the\/ $\mc V$-module\/ $M$, where\/ $M$ is
viewed as a (non-unital) PVA with zero associative product and $\lambda$-bracket.
In particular, $H^2_{\PV}(\mc V,\mc V)$ parameterizes the equivalence classes of first-order deformations of\/ $\mc V$ that preserve the product and the\/ $\mb F[\partial]$-module structure $($cf. \cite{NR67}$)$.
\end{enumerate}
\end{theorem}

\subsection{Relation between LCA cohomology and variational PVA cohomology}
\label{sec:3.4}

\begin{theorem}\label{prop}
Let\/ $R$ be an LCA and consider the universal PVA\/ $\mc V(R)$.
\begin{enumerate}[(a)]
\item
For every module\/ $M$ over the PVA\/ $\mc V(R)$,
 we have a canonical isomorphism of complexes
\begin{equation}\label{eq:wsr}
(C_{\LC}(R,M),d)
\,\stackrel{\sim}{\longrightarrow}\,
(C_{\PV}(\mc V(R),M),d)
\,.
\end{equation}
\item
Let\/ $C\in R$ be such that\/ $\partial C=0$,
let\/ $\bar R=R/\mb FC$ be the corresponding quotient LCA,
and let\/ $\mc V^c(R)=\mc V(R)/\mc V(R)(C-c)$ be the corresponding quotient PVA.
Let\/ $M$ be a module over the PVA\/ $\mc V^c(R)$.
Then, we have natural embeddings of complexes
(cf.\ Proposition \ref{prop:central-ext})
$$
C_{\LC}(\bar R,M) \subset C_{\LC}(R,M)
\,\,\text{ and }\,\,
C_{\PV}(\mc V^c(R),M)\subset C_{\PV}(\mc V(R),M)
\,,
$$
and the isomorphism \eqref{eq:wsr}
restricts to an isomorphism of complexes
\begin{equation}\label{eq:wsrc}
(C_{\LC}(\bar{R},M),d)
\,\stackrel{\sim}{\longrightarrow}\,
(C_{\PV}(\mc V^c(R),M),d)
\,.
\end{equation}
\end{enumerate}
\end{theorem}
\begin{proof}
By definition, an element $Y\in C_{\LC}^n(R,M)$ for $n\geq0$
is a linear map 
$$
Y\colon R^{\otimes n}\to 
M[\lambda_1,\dots,\lambda_n]/\langle\partial+\lambda_1+\dots+\lambda_n\rangle
$$
satisfying the sesquilinearity and symmetry conditions \eqref{eq:sesq}, \eqref{eq:skew}.
Note that the Leibniz rule \eqref{eq:leib}
is symmetric with respect to exchanging $b_i$ and $c_i$.
Moreover, if we plug in the $i$-th position in $Y$
the product $b_i(c_id_i)$ or $(b_ic_i)d_i$,
we get the same answer:
\begin{align*}
& \pm 
(e^{\partial\partial_{\lambda_i}}b_i)(e^{\partial\partial_{\lambda_i}}c_i)
Y_{\lambda_1,\dots,\lambda_n}
(a_1\otimes\cdots\otimes d_i \otimes\cdots\otimes a_n) \\
& \pm 
(e^{\partial\partial_{\lambda_i}}b_i)(e^{\partial\partial_{\lambda_i}}d_i)
Y_{\lambda_1,\dots,\lambda_n}
(a_1\otimes\cdots\otimes c_i \otimes\cdots\otimes a_n) \\
& \pm 
(e^{\partial\partial_{\lambda_i}}c_i)(e^{\partial\partial_{\lambda_i}}d_i)
Y_{\lambda_1,\dots,\lambda_n}
(a_1\otimes\cdots\otimes b_i \otimes\cdots\otimes a_n) 
\,.
\end{align*}
Hence, by the universal property of the symmetric algebra $S(R)$,
the map $Y$ uniquely extends to a linear map
$$
\widehat{Y}\colon\mc V(R)^{\otimes n}\to 
M[\lambda_1,\dots,\lambda_n]/\langle\partial+\lambda_1+\dots+\lambda_n\rangle
\,,
$$
which vanishes when one of its arguments lies in $\mb F1$
and satisfies the Leibniz rules \eqref{eq:leib}.
It is not hard to check, inductively on the polynomial degrees, 
that the resulting map $\widehat{Y}$
still satisfies the sesquilinearity and symmetry conditions \eqref{eq:sesq}, \eqref{eq:skew}.
Hence, $\widehat{Y}$ lies in $C_{\PV}^n(\mc V(R),M)$.
This gives a bijection $C_{\LC}^n(R,M)\to C_{\PV}^n(\mc V(R),M)$,
mapping $Y\mapsto\widehat{Y}$,
thanks to Lemma \ref{lem:id}.
The fact that the differential $d$ defined by \eqref{eq:lca-d}
commutes with taking the restriction to $R^{\otimes n}$
is immediate.
Hence, we have an isomorphism of complexes, as claimed in (a).

To say that $Y$ lies in $C_{\LC}^n(\bar{R},M)\subset C_{\LC}(R,M)$
is the same as saying that $Y$ vanishes when one of its arguments is $C$.
But by Lemma \ref{lem:id}(b)
this is the same as saying that $\widehat{Y}$
vanishes when one of its arguments lies in $\mc V(R)(C-c)$.
In turn, this is equivalent to say that $\widehat{Y}$ lies in $C_{\PV}^n(\mc V^c(R),M)$.
Claim (b) follows.
\end{proof}

We have the following analogue of Proposition \ref{prop:central-ext}.
\begin{proposition}\label{prop:central-ext-pva}
Let\/ $\mc V$ be a PVA
and\/ $C\in\mc V$ be such that\/ $\partial C=0$.
Consider the quotient PVA\/ $\mc V^c=\mc V/\mc V(C-c)$,
where\/ $c\in\mb F$.
Let\/ $M$ be a module over the PVA\/ $\mc V$
such that\/ $Cm=cm$ for every\/ $m\in M$.
Then:
\begin{enumerate}[(a)]
\item
We have canonical linear maps
\begin{equation}\label{eq:w-maps-pva}
\begin{split}
& C_{\PV}^1(\mc V^c,M)\hookrightarrow C_{\PV}^1(\mc V,M)
\,,\\
& C_{\PV}^n(\mc V^c,M)\stackrel{\sim}{\longrightarrow}C_{\PV}^n(\mc V,M)
\,\text{ for }\, n\neq1
\,.
\end{split}
\end{equation}
\item
We have canonical linear maps
\begin{equation}\label{eq:h-maps-pva}
\begin{split}
& H_{\PV}^1(\mc V^c,M)\hookrightarrow H_{\PV}^1(\mc V,M)
\,,\\
& H_{\PV}^2(\mc V^c,M)\twoheadrightarrow H_{\PV}^2(\mc V,M)
\,,\\
& H_{\PV}^n(\mc V^c,M)\stackrel{\sim}{\longrightarrow}H_{\PV}^n(\mc V,M)
\,\text{ for }\, n\neq1,2
\,.
\end{split}
\end{equation}
\item
Let\/ $U=\ker(\partial|_M)$,
and assume that, as a differential algebra, $\mc V\simeq \mb F[C]\otimes\mc V^c$.
Then
\begin{equation}\label{eq:H01-pva}
\begin{split}
& \bigl( \dim H_{\PV}^1(\mc V,M)-\dim H_{\PV}^1(\mc V^c,M) \bigr) \\
& +\bigl( \dim H_{\PV}^2(\mc V^c,M) - \dim H_{\PV}^2(\mc V,M) \bigr)
=
\dim U
\,.
\end{split}
\end{equation}
Note that, in the left-hand side of \eqref{eq:H01-pva}, both summands are non-negative by part (b).
\end{enumerate}
\end{proposition}
\begin{proof}
The proof is the same as the proof of Proposition \ref{prop:central-ext}.
For part (c) we use the fact that, under the assumption that
$\mc V$ splits as $\mc V\simeq \mb F[C]\otimes\mc V^c$,
we have (cf. \eqref{0111:eq1})
\begin{equation}\label{eq:w-maps-pva2}
C_{\PV}^1(\mc V,M)
=
C_{\PV}^1(\mc V^c,M)\oplus U
\,,
\end{equation}
by the Leibniz rule \eqref{eq:deriv} and Lemma \ref{lem:id}.
\end{proof}

\begin{proposition}\label{prop:corollary}
Let\/ $\bar R$ be a Lie conformal algebra that is free as an\/ $\mb F[\partial]$-module,
and let\/ $R=\bar R\oplus\mb FC$ be its LCA central extension
by an element\/ $C$ such that\/ $\partial C=0$.
Consider the universal enveloping PVA\/ $\mc V=\mc V(R)$ and its quotient\/ $\mc V^c=\mc V/\mc V(C-c)$,
for\/ $c\in\mb F$.
\begin{enumerate}[(a)]
\item
If the central extension\/ $R$ of\/ $\bar R$ is trivial, then
\begin{equation}\label{eq:cor1}
\dim H^n_{\PV}(\mc V,\mc V^c)
=
\dim H^n_{\PV}(\mc V^c,\mc V^c)+\delta_{n,1}
\,,\qquad n\geq0
\,.
\end{equation}
\item
If the central extension\/ $R$ of\/ $\bar R$ is nontrivial, then
\begin{equation}\label{eq:cor2}
\dim H^n_{\PV}(\mc V,\mc V^c)
=
\dim H^n_{\PV}(\mc V^c,\mc V^c)-\delta_{n,2}
\,,\qquad n\geq0
\,.
\end{equation}
\end{enumerate}
\end{proposition}
\begin{proof}
Note that,
as differential algebras, $\mc V\simeq\mb F[C]\otimes\mc V^c$
and $\mc V^c\simeq S(\bar R)$ is an algebra of differential polynomials.
In particular, 
$U=\ker(\partial|_{\mc V^c})=\mb F1$.
We can then apply Proposition \ref{prop:central-ext-pva}
to get that either \eqref{eq:cor1} or \eqref{eq:cor2} holds.
Moreover, by \eqref{eq:w-maps-pva} and \eqref{eq:w-maps-pva2},
$C^n_{\PV}(\mc V^c,\mc V^c)$ and $C^n_{\PV}(\mc V,\mc V^c)$ differ only at $n=1$,
and
$$
C^1_{\PV}(\mc V,\mc V^c)
=
C^1_{\PV}(\mc V^c,\mc V^c)\oplus\mb FZ
\,,
$$
where $Z$ is uniquely defined by $Z_\lambda(\bar R)=0$ and $Z_\lambda(C)=1$.
For $a,b\in\bar R$, we have
$$
(dZ)_{\lambda,\mu}(a\otimes b)
=
Z_{\lambda+\mu}([a_\lambda b])
=
\alpha_\lambda(a\otimes b)
+\langle\partial+\lambda+\mu\rangle
\,,
$$
where $\alpha_\lambda\colon \bar R\otimes\bar R\to\mb F[\lambda]$
is the $2$-cocycle defining the central extension $R$:
$$
[a_\lambda b]^R=[a_\lambda b]^{\bar R}+\alpha_\lambda(a\otimes b)C
\,,\qquad a,b\in\bar R
\,.
$$
The claim follows.
\end{proof}

\subsection{Basic PVA cohomology complex}\label{sec:bas-pva}

Now we will review the basic PVA cohomology complex introduced in \cite{DSK13}.
The discussion in this subsection will be similar to the case of basic LCA cohomology from Section \ref{sec:bas-lca}.

Let $\mc V$ be a PVA and $M$ be a $\mc V$-module. 
We let $\widetilde{C}_{\PV}^n(\mc V,M)$ be the vector superspace,
with parity $\bar p$ induced by the opposite parities $\bar p=1-p$ of $\mc V$ and $M$, 
consisting of all linear maps
\begin{equation}\label{eq:basicY}
\widetilde Y\colon \mc V^{\otimes n}\to M[\lambda_1,\dots,\lambda_n] \,,
\end{equation}
satisfying the sesquilinearity conditions \eqref{eq:sesq},
the symmetry conditions \eqref{eq:skew}
and the Leibniz rules \eqref{eq:leib}
(where all the equations are now in the space $M[\lambda_1,\dots,\lambda_n]$).
Elements of $\widetilde{C}_{\PV}^n(\mc V,M)$ are called 
\emph{basic} $n$-\emph{cochains} of $\mc V$ with coefficients in $M$.

\begin{remark}\label{rem:gen}
Suppose that the PVA $\mc V$ is a superalgebra of differential polynomials in the even or odd variables $u_i$, where $i$ is in some (possibly infinite) index set $I$.
Then, by sesquilinearity and Leibniz rules,
an element $\widetilde{Y}\in\widetilde{C}_{\PV}^n(\mc V,M)$
is uniquely determined by its values on the generators $u_i$.
In other words, $\widetilde{Y}$ is uniquely determined by the (arbitrary)
collection of polynomials
\begin{equation}\label{eq:polynY}
\widetilde{Y}_{\lambda_1,\dots,\lambda_n}(u_{i_1}\otimes\dots\otimes u_{i_n})
\,\in\,M[\lambda_1,\dots,\lambda_n]
\,,
\end{equation}
for $i_1,\dots,i_n\in I$,
satisfying only the symmetry condition \eqref{eq:skew}.
\end{remark}

The same formula \eqref{eq:tilde-partial}, as in the LCA case, defines an action of $\partial$ 
on the spaces $\widetilde{C}_{\PV}^n(\mc V,M)$.
Moreover, $\partial$ is injective on $\widetilde{C}_{\PV}^n(\mc V,M)$ for $n\ge1$ by  Lemma \ref{lem:partial-inj}.
Recall the map $\pi$ defined by \eqref{eq:piM}. Then we have the following analogue of Lemma \ref{lem:Ytilde-lca}.


\begin{lemma}\label{lem:Ytilde}
We have a well-defined linear map 
\begin{equation}\label{eq:piYtilde3}
\widetilde{C}_{\PV}^n(\mc V,M)\to {C}_{\PV}^n(\mc V,M)
\,,\qquad
\widetilde Y\mapsto\pi\circ\widetilde Y
\,.
\end{equation}
\begin{enumerate}[(a)]
\item
The map \eqref{eq:piYtilde3} has kernel $\partial\widetilde{C}_{\PV}^n(\mc V,M)$.
Hence \eqref{eq:piYtilde3} induces an injective linear map
$$
\widetilde{C}_{\PV}^n(\mc V,M)
/\partial\widetilde{C}_{\PV}^n(\mc V,M)
\,\hookrightarrow\,
{C}_{\PV}^n(\mc V,M)
\,.
$$
\item
Suppose that, as a differential superalgebra, $\mc V$
is a superalgebra of differential polynomials in even or odd variables.
Then \eqref{eq:piYtilde3} is surjective for all\/ $n\ge0$.
Hence, we get an isomorphism
$$
\widetilde{C}_{\PV}^n(\mc V,M)
/\partial\widetilde{C}_{\PV}^n(\mc V,M)
\,\stackrel{\sim}{\longrightarrow}\,
{C}_{\PV}^n(\mc V,M)
\,.
$$
\end{enumerate}
\end{lemma}
\begin{proof}
Clearly, if $\widetilde Y$ satisfies equations \eqref{eq:sesq},
\eqref{eq:skew} and \eqref{eq:leib}
in $M[\lambda_1,\dots,\lambda_n]$,
so does $\pi\circ\widetilde Y$ in
$M[\lambda_1,\dots,\lambda_n]/\langle\partial+\lambda_1+\dots+\lambda_n\rangle$.
Hence, composing with $\pi$
defines a linear map
$\widetilde{C}_{\PV}^n(\mc V,M)\to {C}_{\PV}^n(\mc V,M)$.
Claim (a) is proved in \cite[Proposition 7.3(c)]{DSK13}.

Claim (b) is obvious for $n=0$, hence to prove it we may assume that $n\ge1$.
%
Let $Y\in{C}_{\PV}^n(\mc V,M)$.
We construct its preimage $\widetilde{Y}\in\widetilde{C}_{\PV}^n(\mc V,M)$
as follows.
For every $\ell=1,\dots,n$, we have the identification
\begin{equation}\label{eq:iota-ell}
\iota_\ell
\colon
M[\lambda_1,\dots,\lambda_n]/\langle\partial+\lambda_1+\dots+\lambda_n\rangle
\stackrel{\sim}{\longrightarrow}
M[\lambda_1,\stackrel{\ell}{\check{\dots}},\lambda_n]
\subset
M[\lambda_1,\dots,\lambda_n]
\,,
\end{equation}
obtained by replacing $\lambda_\ell$
by $-\lambda_1-\stackrel{\ell}{\check{\dots}}-\lambda_n-\partial$.
Obviously, we have
\begin{equation}\label{eq:pi-iota}
\pi\circ\iota_\ell=\id
\,\,\text{ for every }\, \ell=1,\dots,n
\,.
\end{equation}

By assumption, $\mc V = \mb F \bigl[ u_i^{(k)} \bigr]$
is a superalgebra of differential polynomials.
We let (cf.\ \cite[Remark 6.6]{DSK13}):
\begin{equation*}
\widetilde{Y}_{\lambda_1,\dots,\lambda_n}(u_{i_1}\otimes\dots\otimes u_{i_n})
=
\frac1n
\sum_{\ell=1}^n
\iota_\ell
\big({Y}_{\lambda_1,\dots,\lambda_n}(u_{i_1}\otimes\dots\otimes u_{i_n})\big)
\,\in\,M[\lambda_1,\dots,\lambda_n]
\,.
\end{equation*}
It is immediate to check that 
$\widetilde{Y}$
satisfies the symmetry conditions \eqref{eq:skew},
since $Y$ does. By Remark \ref{rem:gen},
$\widetilde{Y}$ extends uniquely to a linear map
$\widetilde{Y}\colon \mc V^{\otimes n}\to M[\lambda_1,\dots,\lambda_n]$
using the sesquilinearity conditions and the Leibniz rules.
Hence, $\widetilde{Y}$ is a well-defined element of
$\widetilde{C}_{\PV}^n(\mc V,M)$.
By equation \eqref{eq:pi-iota}, $\pi\circ \widetilde{Y}$ and $Y$
have the same value
on all $u_{i_1}\otimes\dots\otimes u_{i_n}$;
therefore they must coincide: $\pi\circ\widetilde{Y}=Y$.
This proves surjectivity. 
\end{proof}


The basic PVA cohomology \emph{differential} $\widetilde d\colon\widetilde{C}_{\PV}^{n}(\mc V,M)\to\widetilde{C}_{\PV}^{n+1}(\mc V,M)$ 
is defined again by \eqref{eq:lca-d}, viewed as an equation in $M[\lambda_0,\dots,\lambda_n]$.
Then Lemma \ref{lem:basic-d} holds as well. We let
\begin{equation}\label{eq:tildeW-pv}
\widetilde C_{\PV}(\mc V,M)=\bigoplus_{n\geq0} \widetilde C_{\PV}^n(\mc V,M)
\,.
\end{equation}

\begin{definition}\label{def:pva-bas}
Given a module $M$ over the PVA $\mc V$,
the cohomology of the complex $\bigl(\widetilde C_{\PV}(\mc V,M),\widetilde d\bigr)$
is called the \emph{basic PVA cohomology} of $\mc V$ with coefficients in $M$:
\begin{equation}\label{eq:h-pva-bas}
\begin{split}
\widetilde H_{\PV}(\mc V,M)
&=
\bigoplus_{n\geq0}
\widetilde H_{\PV}^n(\mc V,M)
\,,\\
\widetilde H_{\PV}^n(\mc V,M)
&=
\ker\big(\widetilde d|_{\widetilde C_{\PV}^{n}(\mc V,M)}\big)
/
\widetilde d\big(\widetilde C_{\PV}^{n-1}(\mc V,M)\big)
\,.
\end{split}
\end{equation}
\end{definition}
\begin{remark}\label{rem:lca-bas-pva}
Note that, by definition, the basic PVA complex $\bigl(\widetilde C_{\PV}(\mc V,M), \widetilde d\bigr)$ is a subcomplex of the
basic LCA complex $\bigl(\widetilde C_{\LC}(\mc V,M), \widetilde d\bigr)$, where in the latter, $\mc V$ is viewed as an LCA.
\end{remark}

Having a short exact sequence of complexes
\begin{equation}\label{eq:tildeW-ses}
0 \to \partial\widetilde C_{\PV}(\mc V,M) \to \widetilde C_{\PV}(\mc V,M) \to \widetilde C_{\PV}(\mc V,M) / \partial\widetilde C_{\PV}(\mc V,M) \to 0
\end{equation}
leads to a long exact sequence of cohomology. 
Under the assumptions of Lemma \ref{lem:Ytilde}
we obtain the long exact sequence
\begin{equation}\label{eq:tildeW-les}
\begin{split}
0 \to H^0\big(\partial\widetilde C_{\PV}(\mc V,M)\big)
&\to \widetilde H_{\PV}^0(\mc V,M) \to H_{\PV}^0(\mc V,M) \to
\\
\to H^1\big(\partial\widetilde C_{\PV}(\mc V,M)\big) 
&\to \widetilde H_{\PV}^1(\mc V,M) \to H_{\PV}^1(\mc V,M) \to
\\
\to H^2\big(\partial\widetilde C_{\PV}(\mc V,M)\big) 
&\to \widetilde H_{\PV}^2(\mc V,M) \to H_{\PV}^2(\mc V,M) \to\cdots
\,.
\end{split}
\end{equation}
By Lemmas \ref{lem:partial-inj} and \ref{lem:basic-d}(b), we have
\begin{equation}\label{eq:tildeW-les2}
H^n\big(\partial\widetilde C_{\PV}(\mc V,M)\big)
\simeq \widetilde H_{\PV}^n(\mc V,M) \,, \qquad n\ge 0 \,.
\end{equation}
Hence, as an immediate consequence, 
we obtain from \eqref{eq:tildeW-les} the following. 

\begin{proposition}\label{prop:Ytilde}
Assume that, as a differential superalgebra, $\mc V$
is a superalgebra of differential polynomials in even or odd variables,
and that\/ $\widetilde H_{\PV}^n(\mc V,M)=0$ for all\/ $n\ge0$. Then
$H_{\PV}^n(\mc V,M)=0$ for all\/ $n\ge0$.
\end{proposition}

For $a\in\mc V$ and $\widetilde Y\in\widetilde C_{\PV}^n(\mc V,M)$, we define $a_\la\widetilde Y$ and $\iota_\la(a) \widetilde Y$
by the same formulas \eqref{eq:lca-modstr} and \eqref{eq:lca-iota2}, respectively, as in the LCA case.
Obviously, if $\widetilde Y$ satisfies the Leibniz rule \eqref{eq:leib}, then so does $\iota_\la(a) \widetilde Y$. By Cartan's formula
\eqref{eq:cartan-lc}, the same holds for $a_\la\widetilde Y$ (or this can be easily checked directly).
Thus, we obtain an LCA action of $\mc V$ (viewed as an LCA) on the basic PVA complex 
$\widetilde C_{\PV}(\mc V,M)$,
where the $\lambda$-action is by formal power series in $\lambda$. 
As before, this induces a trivial action on the basic
PVA cohomology $\widetilde H_{\PV}(\mc V,M)$; see Corollary \ref{ccartan}.
To summarize, we get the following PVA analogue of Corollary \ref{ccartan} and Proposition \ref{pcartan2}.
\begin{proposition}\label{ccartan-pva}
\begin{enumerate}[(a)]
\item
The\/ $\lambda$-action\/ $a_\lambda\widetilde Y$ of\/ $\mc V$ on the basic complex\/ $\widetilde C_{\PV}(\mc V,M)$ commutes with the differential\/ $\widetilde d$, and it induces a trivial action on its cohomology. 
\item
We have a Lie algebra action of\/ $\mc V/\partial \mc V$ on\/ $C_{\PV}(\mc V,M)$,
given by the zero modes $($of parity $p(a)){:}$
\begin{equation}\label{eq:pva-modstr}
\begin{split}
(a_{(0)} Y&)_{\lambda_1,\dots,\lambda_n}
(a_1\otimes\cdots\otimes a_n) 
=
a_{(0)}
\bigl(Y_{\lambda_1,\dots,\lambda_n}
(a_1\otimes\cdots\otimes a_n) \bigr)
\\
& +
\sum_{i=1}^n
(-1)^{\delta_i}
Y_{\lambda_1,\dots,\lambda_n}
(a_1\otimes\cdots\otimes (a_{(0)}a_i)\otimes\cdots\otimes a_n) 
\,,
\end{split}
\end{equation}
where\/ $\delta_i$ are as in \eqref{eq:deltai}.
We also have well-defined contraction operators 
$\iota_0(a) \colon C_{\PV}^n(\mc V,M) \to C_{\PV}^{n-1}(\mc V,M)$,
of parity $\bar p(a)$, given by{\rm{:}} 
\begin{equation}\label{eq:pva-iota3}
\begin{split}
\bigl(\iota_0(a) Y &\bigr)_{\lambda_1,\dots,\lambda_{n-1}}
(a_1\otimes\cdots\otimes a_{n-1}) 
\\
&= (-1)^{\bar p(a)\bar p(Y)} \,
Y_{0,\lambda_1,\dots,\lambda_{n-1}}
(a\otimes a_1 \otimes\cdots\otimes a_{n-1})
\,.
\end{split}
\end{equation}
The following Cartan's formula holds on\/ $C_{\PV}(\mc V,M)${\rm{:}}
\begin{equation}\label{eq:cartan2-pva}
a_{(0)} = [\iota_0(a), d]
:= \iota_0(a) \, d - (-1)^{\bar p(a)} d \, \iota_0(a)
\,, \qquad a\in \mc V\,.
\end{equation}
Thus, the action of the Lie algebra\/ $\mc V/\partial \mc V$ on\/ $C_{\PV}(\mc V,M)$ 
by the zero modes \eqref{eq:pva-modstr} commutes with the differential\/ $d$ and induces the
trivial action on the cohomology\/ $H_{\PV}(\mc V,M)$.
\end{enumerate}
\end{proposition}

\subsection{Virasoro element and conformal weights}
\label{sec:vir}

%


The following notion will play an important role through the rest of the paper.

\begin{definition}\label{def:vir}
A \emph{Virasoro element} in a PVA $\mc V$ is an even element $L\in\mc V$
such that 
$$
[L_\lambda L]=(\partial+2\lambda)L+\frac{c}{12}\lambda^3
\,,\,\,\text{ for some } c\in\mb F \,\,(\text{the \emph{central charge} of } L)
\,,
$$
(cf.\ Examples \ref{ex:virasoro-lca} and \ref{ex:virasoro-pva}),
and such that 
\begin{equation}\label{eq:l1}
L_{(0)}
:=[L_\lambda\,\cdot\,]\big|_{\lambda=0}
=\partial
\,,\,\,\,\text{ and }\,\,
L_{(1)}
:=\frac{d}{d\lambda}[L_\lambda\,\cdot\,]\big|_{\lambda=0}
\,\in\End\mc V
\,\text{ is diagonalizable.}
\end{equation}
A PVA $\mc V$ is called \emph{conformal} if it is endowed 
with a Virasoro element $L\in\mc V$.
One also says that $a\in\mc V$ 
has \emph{conformal weight} $\Delta(a)\in\mb F$
if it is an eigenvector of $L_{(1)}$ of eigenvalue $\Delta(a)$.
A PVA-module $M$ over $\mc V$ is called \emph{conformal}
with respect to the Virasoro element $L\in\mc V$ if
\begin{equation}\label{eq:m1}
L_{(0)}^M
:=(L_\lambda\,\cdot\,)\big|_{\lambda=0}
=\partial^M
\,\,,\,\,\text{ and }\,\,
L_{(1)}^M
:=\frac{d}{d\lambda}(L_\lambda\,\cdot\,)\big|_{\lambda=0}
\,\in\End M
\,\text{ is diagonalizable.}
\end{equation}
As before, one says that $m\in M$ 
has \emph{conformal weight} $\Delta(m)\in\mb F$
if it is an eigenvector of $L_{(1)}^M$ of eigenvalue $\Delta(m)$.
 \end{definition}
%
%
We also extend the notion of conformal weight to the spaces of polynomials
$\mc V[\lambda]$ and $M[\lambda]$ by letting $\Delta(a\lambda^n)=\Delta(a)+n$
(i.e., we assign to $\lambda$ conformal weight $1$ and extend in the obvious way).
In other words, the conformal weights in $M[\lambda]$ are the eigenvalues
of the operator
\begin{equation}\label{eq:euler}
E:=L_{(1)}^M+\lambda\frac{d}{d\lambda}
\,.
\end{equation}

Throughout the remainder of this subsection,
we let $\mc V$ be a conformal PVA and $M$ be a conformal $\mc V$-module.
\begin{lemma}\label{lem:conf-weight}
Let\/ $a\in\mc V$ and\/ $m\in M$ have conformal weights\/ $\Delta(a)$ and\/ $\Delta(m)$.
Then{\rm:}
\begin{enumerate}[(a)]
\item
the unit element\/ $1\in\mc V$ has conformal weight\/ $\Delta(1)=0;$
\item
$\Delta(\partial a)=\Delta(a)+1$, and\/
$\Delta(\partial m)=\Delta(m)+1;$
\item
$\Delta(a_\lambda m)=\Delta(a)+\Delta(m)-1;$
\item
$\Delta(am)=\Delta(a)+\Delta(m)$.
\end{enumerate}
\end{lemma}
\begin{proof}
Part (a) is obvious, and (b) follows from the sesquilinearities L1 and M1.
For part (c), by the Jacobi identity M2, we have
\begin{align*}
& L_{(1)}^M(a_\mu m)
=
\frac{d}{d\lambda}([L_\lambda a]_{\lambda+\mu} m)\big|_{\lambda=0}
+\frac{d}{d\lambda}a_\mu (L_\lambda m)\big|_{\lambda=0} \\
& =
(L_{(1)}a)_{\mu} m
+\frac{d}{d\mu} (L_{(0)}a)_{\mu} m
+a_\mu (L_{(1)}^M m) \\
& =
(\Delta(a)+\Delta(m)) a_{\mu} m
-\frac{d}{d\mu} (\mu \, a_{\mu} m) \\
& =
(\Delta(a)+\Delta(m)-1) a_{\mu} m
-\mu\frac{d}{d\mu}(a_\mu m)
\,.
\end{align*}
Then (c) follows from \eqref{eq:euler}.
Claim (d) is an immediate consequence of the Leibniz rule M3.
\end{proof}


\begin{lemma}\label{lem:E}
The linear operator
\begin{equation}\label{eq:euler-k}
E
:=
L_{(1)}^M+\sum_{i=1}^n\lambda_i\frac{d}{d\lambda_i}
\,,
\end{equation}
is a diagonalizable even endomorphism of\/ $M[\lambda_1,\dots,\lambda_n]$,
which leaves invariant
the image of the operator\/ $\partial+\lambda_1+\dots+\lambda_n$.
Hence, it induces a diagonalizable even endomorphism,
still denoted by\/ $E$, on the quotient space
$$
M[\lambda_1,\dots,\lambda_n]/\langle\partial+\lambda_1+\dots+\lambda_n\rangle
\,.
$$
\end{lemma}
\begin{proof}
By Lemma \ref{lem:conf-weight}(b),
we have the following commutation rule
\begin{equation}\label{eq:euler-comm}
E
\circ
(\partial+\lambda_1+\dots+\lambda_n)
=
(\partial+\lambda_1+\dots+\lambda_n)
\circ 
(E+1)
\,.
\end{equation}
The claim follows.
\end{proof}

We will call the operator $E$ given by  \eqref{eq:euler-k} the \emph{energy operator}, and its eigenvalues will be called
\emph{conformal weights}. 
We denote by $\Delta(P(\lambda_1,\dots,\lambda_n))$ the eigenvalue
of the eigenvector $P(\lambda_1,\dots,\lambda_n)$.

Consider the cohomology complex $C_{\PV}(\mc V,M)$
and define the linear operator $E$ on it by
\begin{equation}\label{eq:euler2}
\begin{split}
(EY)_{\lambda_1,\dots,\lambda_n}(a_1&\otimes\dots\otimes a_n)
=
(E+n) \big(
Y_{\lambda_1,\dots,\lambda_n}(a_1\otimes\dots\otimes a_n)
\big) \\
& -
\sum_{i=1}^n
Y_{\lambda_1,\dots,\lambda_n}(a_1\otimes\dots\otimes L_{(1)}a_i  \otimes\dots\otimes a_n)
\,,
\end{split}
\end{equation}
which we will call again the \emph{energy operator}.
%
%
By Lemma \ref{lem:E}, $E$ is diagonalizable on $C_{\PV}(\mc V,M)$.
%
As before, we call \emph{conformal weights} the eigenvalues of the energy operator $E$
in \eqref{eq:euler2},
and we denote by $\Delta(Y)$ the eigenvalue
of the eigenvector $Y\in C_{\PV}^n(\mc V,M)$.
By \eqref{eq:euler2}, we have:
\begin{equation}\label{eq:DY}
\Delta(Y_{\lambda_1,\dots,\lambda_n}(a_1\otimes\dots\otimes a_n))
=
\Delta(Y)+\Delta(a_1)+\dots+\Delta(a_n)-n
\,.
\end{equation}

\begin{lemma}\label{lem:Ed}
The energy operator\/ $E\in\End C_{\PV}(\mc V,M)$, defined by \eqref{eq:euler2},
commutes with the differential\/ $d\colon C_{\PV}^{n}(\mc V,M)\to C_{\PV}^{n+1}(\mc V,M)$
in \eqref{eq:lca-d}.
As a consequence, $E$ induces a diagonalizable endomorphism
in cohomology:
\begin{equation*}
E\,\in\,\End H_{\PV}^{n}(\mc V,M)
\,.
\end{equation*}
\end{lemma}
\begin{proof}
Since the linear operator $E\in\End C_{\PV}^n(\mc V,M)$ 
is diagonalizable, 
it suffices to prove that if $Y\in C_{\PV}^n(\mc V,M)$ is an eigenvector of $E$
with eigenvalue $\Delta(Y)$,
then $dY\in C_{\PV}^{n+1}(\mc V,M)$ is also an $E$-eigenvector
with the same eigenvalue.
Let then $a_0,\dots,a_n\in\mc V$ have conformal weights $\Delta(a_0),\dots,\Delta(a_n)$.
By equation \eqref{eq:DY} and Lemma \ref{lem:conf-weight}, we have:
\begin{equation}\label{eq:E1}
\begin{split}
\Delta\big(
{a_i}_{\lambda_i}
& Y_{\lambda_0,\stackrel{i}{\check{\dots}},\lambda_n}
(a_0\otimes\stackrel{i}{\check{\dots}}\otimes a_n) 
\big)
=
\Delta(a_i)
+\Delta(Y_{\lambda_0,\stackrel{i}{\check{\dots}},\lambda_n}
(a_0\otimes\stackrel{i}{\check{\dots}}\otimes a_n) )
-1 \\
& =
\Delta(Y)+\Delta(a_0)+\dots+\Delta(a_n)-n-1
\,,
\end{split}
\end{equation}
and
\begin{equation}\label{eq:E2}
\begin{split}
\Delta\big(
& Y_{\lambda_i+\lambda_j,\lambda_0,
\stackrel{i}{\check{\dots}}\stackrel{j}{\check{\dots}},\lambda_n}
([{a_i}_{\lambda_i}a_j]\otimes a_0\otimes
\stackrel{i}{\check{\dots}}\stackrel{j}{\check{\dots}}
\otimes a_n) 
\big)
=
\Delta(Y)+\Delta([{a_i}_{\lambda_i}a_j]) \\
& +\Delta(a_0)+\stackrel{i}{\check{\dots}}\stackrel{j}{\check{\dots}}+\Delta(a_n)-n 
=
\Delta(Y)+\Delta(a_0)+\dots+\Delta(a_n)-n-1
\,.
\end{split}
\end{equation}
Combining \eqref{eq:E1} and \eqref{eq:E2},
and recalling the definition \eqref{eq:lca-d} of the differential $d$,
we get that
$$
\Delta((dY)_{\lambda_0,\dots,\lambda_n}(a_0\otimes\dots\otimes a_n))
-\Delta(a_0)-\dots-\Delta(a_n)+n+1
$$
is well defined, it is independent of the 
$L_{(1)}$-eigenvectors $a_0,\dots,a_n$,
and it is equal to $\Delta(Y)$.
Recalling \eqref{eq:DY},
this precisely means that $dY$ is an eigenvector of $E$ of eigenvalue $\Delta(Y)$,
i.e., $dY$ has conformal weight $\Delta(dY)=\Delta(Y)$.
\end{proof}

As before, we call \emph{conformal weights} the eigenvalues of $E$ in $H_{\PV}(\mc V,M)$
and we denote by $\Delta([Y])$ the conformal weight
of the cohomology class $[Y]\in H_{\PV}^n(\mc V,M)$.

The following result will be the main tool,
in the next Section \ref{sec:com-pva},
for computing the variational PVA cohomology in all the examples considered.
\begin{theorem}\label{thm:Delta}
Let\/ $\mc V$ be a conformal PVA and\/ $M$ be a conformal\/ $\mc V$-module.
Assume that, as a differential superalgebra, $\mc V$
is a superalgebra of differential polynomials in even or odd variables.
Then the energy operator\/ $E\in\End H_{\PV}(\mc V,M)$
is diagonalizable with
only eigenvalues\/ $0$ and\/ $1$.
\end{theorem}

In the proof of the theorem, we will use the LCA action $L_\la$ of the element $L\in\mc V$ on the basic PVA complex $\widetilde{C}_{\PV}(\mc V,M)$; see Section \ref{sec:bas-pva}. 
We define the energy operator on $\widetilde{C}_{\PV}(\mc V,M)$ by
\begin{equation}\label{eq:Etilde}
\widetilde{E}=\frac{d}{d\lambda}L_{\lambda}\big|_{\lambda=0} = L_{(1)}
\,,
\end{equation}
the coefficient of $\lambda$ in the map $L_\lambda$ defined by \eqref{eq:lca-modstr}.
Note that, by Cartan's formula \eqref{eq:cartan-lc},
$\widetilde E$ commutes with the action of $\widetilde d$
(cf.\ Corollary \ref{ccartan}).
\begin{lemma}\label{lem:Etilde}
The energy operator\/ $\widetilde{E}$ is given explicitly by \eqref{eq:euler2}, 
where we replace\/ $E$ with\/ $\widetilde{E}$ and view both sides as elements of\/ $M[\la_1,\dots,\la_n]$.
\end{lemma}
\begin{proof}
By \eqref{eq:Etilde} and \eqref{eq:lca-modstr}, we have
\begin{align*}
& (\widetilde E\widetilde Y)_{\lambda_1,\dots,\lambda_n}
(a_1\otimes\cdots\otimes a_n) 
=
\frac{d}{d\lambda}(L_\lambda\widetilde Y)_{\lambda_1,\dots,\lambda_n}
(a_1\otimes\cdots\otimes a_n)\Big|_{\lambda=0} \\
& =
L_{(1)}^M\!
\bigl(\widetilde Y_{\lambda_1,\dots,\lambda_n}\!
(a_1\!\otimes\!\cdots\!\otimes\! a_n) \bigr) 
-
\sum_{i=1}^n
\frac{d}{d\lambda}
\widetilde Y_{\lambda_1,\dots,\lambda+\lambda_i,\dots,\lambda_n}
(a_1\!\otimes\!\cdots [L_{\lambda}a_i] \cdots\!\otimes\! a_n) \Big|_{\lambda=0} \\
& =
L_{(1)}^M\!
\bigl(\widetilde Y_{\lambda_1,\dots,\lambda_n\!}
(a_1\!\otimes\!\cdots\!\otimes\! a_n) \bigr) 
-
\sum_{i=1}^n
\widetilde Y_{\lambda_1,\dots,\lambda_i,\dots,\lambda_n}
(a_1\!\otimes\!\cdots\!\otimes (L_{(1)}a_i) \otimes\!\cdots\!\otimes\! a_n) \\
&\quad -
\sum_{i=1}^n
\frac{d}{d\lambda_i}
\widetilde Y_{\lambda_1,\dots,\lambda_i,\dots,\lambda_n}
(a_1 \otimes\cdots\otimes (\partial a_i) \otimes\cdots\otimes a_n) 
\,.
\end{align*}
By the sesquilinearity condition \eqref{eq:sesq},
the last term above 
becomes
$$
\Bigl(n+ \sum_{i=1}^n
\lambda_i\frac{d}{d\lambda_i}
\Bigr)
\widetilde Y_{\lambda_1,\dots,\lambda_i,\dots,\lambda_n}
(a_1 \otimes\cdots\otimes a_n) \,,
$$
thus proving the claim.
\end{proof}
\begin{lemma}\label{lem:tildeE}
\begin{enumerate}[(a)]
\item
The energy operator\/ $\widetilde{E}$ is diagonalizable on\/ $\widetilde{C}_{\PV}^{n}(\mc V,M)$.
\item
If\/ $\widetilde{Y}\in\widetilde{C}_{\PV}^{n}(\mc V,M)$ is an eigenvector of\/ $\widetilde{E}$
with eigenvalue\/ $\Delta$, then\/ $\partial\widetilde{Y}$ is an eigenvector of\/ $\widetilde{E}$
with eigenvalue\/ $\Delta+1$.
\item
For\/ $\widetilde{Y}\in\widetilde{C}_{\PV}^{n}(\mc V,M)$, we have
\begin{equation}\label{eq:tildeE}
\pi\circ(\widetilde{E}\,\widetilde{Y})
=
E(\pi\circ\widetilde{Y})
\,.
\end{equation}
\item
As a consequence,
if\/ $\widetilde{Y}\in\widetilde{C}_{\PV}^{n}(\mc V,M)$ is an eigenvector of\/ $\widetilde{E}$
with eigenvalue\/ $\Delta$, then\/ $\pi\circ\widetilde{Y}$ is an eigenvector of\/ $E$
with the same eigenvalue\/ $\Delta$.
\end{enumerate}
\end{lemma}
\begin{proof}
All of these claims are immediate consequences of the definitions
and of Lemma \ref{lem:Etilde}.
\end{proof}

\begin{proof}[Proof of Theorem \ref{thm:Delta}]
By Corollary \ref{ccartan}, the energy operator $\widetilde{E}$ 
induces a trivial action on the basic PVA cohomology.
Hence, for any cohomology class $[\widetilde Y]\in\widetilde{H}^{n}_{PV}(\mc V,M)$,
its representative $\widetilde Y\in\widetilde{C}_{\PV}^{n}(\mc V,M)$
is a sum of $\widetilde E$-eigenvectors
and we can pick them of eigenvalue $0$, i.e., $\widetilde E\widetilde Y=0$.

Recall that, by Lemmas \ref{lem:partial-inj} and \ref{lem:basic-d}(b), the map $\partial$ is an isomorphism of complexes
from $\widetilde{C}_{\PV}(\mc V,M)$ to $\partial\widetilde{C}_{\PV}(\mc V,M)$ in degree $\ge1$.
Then, for $n\ge1$, a cohomology class $[\widetilde Z] \in H^{n}(\partial\widetilde{C}_{\PV}(\mc V,M))$
has a representative of the form $\widetilde Z=\partial\widetilde Y$ 
for some $\widetilde Y \in \widetilde{C}_{\PV}^{n}(\mc V,M)$
with $\widetilde{E} \widetilde{Y}=0$. Hence, by Lemma \ref{lem:tildeE}(b), we get $\widetilde{E} \widetilde{Z}=\widetilde{Z}$.

As part of the long exact sequence \eqref{eq:tildeW-les}, we have for each $n\ge1$:
$$
\widetilde H_{\PV}^n(\mc V,M) \xrightarrow{\varphi} H_{\PV}^n(\mc V,M)
\xrightarrow{\psi} H^{n+1}\big(\partial\widetilde W_{\PV}(\mc V,M)\big) \,.
$$
Let us prove that the maps $\varphi$ and $\psi$ are compatible with the actions of 
the energy operators $E$ and $\widetilde E$:
\begin{equation}\label{eq:phipsiE}
\varphi\circ\widetilde E=E\circ\varphi
\,,\qquad
\widetilde E\circ\psi=\psi\circ E
\,.
\end{equation}
By definition, the map $\varphi$ is given by
$$
\varphi([\widetilde Y])=[\pi\circ\widetilde Y]
\,.
$$
Hence, by Lemma \ref{lem:tildeE}(c), we have
$$
(\varphi\circ\widetilde E)([\widetilde Y])
=
\varphi([\widetilde E\widetilde Y])
=
[\pi\circ\widetilde E\widetilde Y]
=
[E(\pi\circ\widetilde Y)]
=
(E\circ\varphi)([\widetilde Y])
\,,
$$
proving the first equation in \eqref{eq:phipsiE}.
Next, recall the definition of the connecting homomorphism $\psi$.
By the surjectivity of $\pi$, any element of $H_{\PV}^n(\mc V,M)$ is of the form
$[\pi\circ\widetilde Y]$ for some $\widetilde Y\in\widetilde C^{n}_{PV}(\mc V,M)$,
and $\widetilde d\widetilde Y$ lies in $\partial\widetilde C^{n+1}_{PV}(\mc V,M)$.
Then,
$$
\psi([\pi\circ\widetilde Y])=[\widetilde d\widetilde Y]
\,.
$$
Again by Lemma \ref{lem:tildeE}(c), we have
\begin{align*}
(\psi\circ E)([\pi\circ\widetilde Y])
& =
\psi([E(\pi\circ\widetilde Y)])
=
\psi([\pi\circ(\widetilde E\widetilde Y)])
=
[\widetilde d(\widetilde E\widetilde Y)] \\
& =
\widetilde E[\widetilde d\widetilde Y]
=
(\widetilde E\circ\psi)([\pi\circ\widetilde Y])
\,,
\end{align*}
proving the second equation in \eqref{eq:phipsiE}.

Now consider an element $[Y] \in H_{\PV}^n(\mc V,M)$ with $EY=\Delta Y$, where 
$\Delta\in\mb F$ and assume that $\Delta\ne 0$ or $1$. Since $\Delta\ne 1$, we have
$\psi([Y])=0$. Hence, $[Y] = \varphi([\widetilde Y])$ is in the image of $\varphi$. But $\Delta\ne0$ implies $[\widetilde Y]=0$; therefore, $[Y] =0$, completing the proof of the theorem. 
\end{proof}

\begin{theorem}\label{thm:Delta2}
Let\/ $\mc V$ be a conformal PVA, which as a differential superalgebra
is a superalgebra of differential polynomials in finitely many even or odd variables
with positive $($rational or real if\/ $\mb R\subset\mb F)$ conformal weights.
Let\/ $M$ be a conformal\/ $\mc V$-module, which
is finitely generated as a module over the differential superalgebra\/ $\mc V$.
Then\/ $\dim H_{\PV}^n(\mc V,M) < \infty$
for all\/ $n\ge0$.
\end{theorem}
\begin{proof}
By assumption, as a differential superalgebra,
$$
\mc V = \mb F\bigl[u_i^{(k)} \,\big|\, i=1,\dots,N,\,k\in\mb Z_+\bigr]
\,, \qquad u_i^{(k)} = \partial^k u_i \,.
$$
Let $\Delta_{\mc V} \in\mb Q$ (or $\mb R$) be such that
$0 < \Delta(u_i) \le \Delta_{\mc V}$ for all $i=1,\dots,N$.
Let $\{m_1,\dots,m_L\}$ be a set of generators of $M$ as a module over the differential superalgebra $\mc V$, which are eigenvectors of $L^M_{(1)}$.
Every vector in $M$ is a linear combination of monomials of the form
$$
m = u_{i_1}^{(k_1)} \cdots u_{i_s}^{(k_s)} m_j \,,
\qquad 0\le s \,, \; 1\le j\le L \,, \; 0\le k_t \,, \; 1\le i_t \le N \;\; (1\le t \le s) \,.
$$
By Lemma \ref{lem:conf-weight}, we have
$$
\Delta(m) = \Delta(m_j) + \sum_{t=1}^s (\Delta(u_{i_t}) + k_t) \,.
$$
This implies that $\dim M_\delta < \infty$ for every $\delta\in\mb F$, where $M_\delta$ is the span of all vectors $m\in M$ such that $\delta-\Delta(m)$
is a positive (rational or real) number.

By Lemma \ref{lem:id}(a), any $n$-cocycle $Y\in C_{\PV}^{n}(\mc V, M)$ is uniquely determined by its values on the generators:
$$
Y_{\la_1,\dots,\la_n}^{i_1,\dots,i_n} :=
Y_{\la_1,\dots,\la_n}(u_{i_1} \otimes\dots\otimes u_{i_n}) 
\in M[\la_1,\dots,\la_n] / \langle \partial+\la_1+\dots+\la_n \rangle
\,.
$$
By \eqref{eq:DY} and Theorem \ref{thm:Delta}, we can replace $Y$ with an equivalent cocycle (denoted again $Y$) such that $\Delta(Y)=0$ or $1$,
hence 
$$
\Delta\bigl( Y_{\la_1,\dots,\la_n}^{i_1,\dots,i_n} \bigr)
= \Delta(Y) + \Delta(u_{i_1}) +\dots+ \Delta(u_{i_n}) - n
\le n (\Delta_{\mc V}-1) + 1 \,. 
$$
Thus,
$$
Y_{\la_1,\dots,\la_n}^{i_1,\dots,i_n} 
\in M_{n (\Delta_{\mc V}-1) + 1}[\la_1,\dots,\la_n] / \langle \partial+\la_1+\dots+\la_n \rangle
\,.
$$
Hence, the space of all such $n$-cocycles is finite dimensional.
\end{proof}

\section{Computations of variational PVA cohomology}\label{sec:com-pva}

In this section, we compute the variational PVA cohomology of several examples of Poisson vertex algebras.
For each of them, we first show that the PVA is conformal (see Definition \ref{def:vir}) 
and then use Theorem \ref{thm:Delta}.

\subsection{Cohomology of the free superboson PVA}\label{sec:coh-bos}

Consider the free superboson PVA $\mc B_{\mf h}$ introduced in Example \ref{ex:boson-pva}.
We shall denote by $\Pi\mf h$ the vector superspace $\mf h$ with reversed parity.
Recall that, by assumption, the bilinear form $(\cdot\,|\,\cdot)$ is supersymmetric
and $\mf h_{\bar0}\perp\mf h_{\bar1}$.
This implies 
\begin{equation}\label{eq:symmetry}
(a|b)=(-1)^{p(a)} (b|a)
\,,\qquad a,b\in\mf h\,.
\end{equation}
Moreover, since $(\cdot|\cdot)$ is nondegenerate, 
it induces an isomorphism of vector superspaces $\mf h^* \simeq \mf h$. 

Let $\{u_1,\dots,u_N\}$ be a basis for $\mf h$ homogeneous with respect to parity, 
and $\{u^1,\dots,u^N\}$ be its dual basis, 
so that $(u_i|u^j)=\delta_i^j$. 
Then for every $a\in\mf h$, we have
\begin{equation}\label{eq:dualb}
a = \sum_{j=1}^N (u_j | a) u^j = \sum_{j=1}^N (a | u^j) u_j
\,.
\end{equation}

\begin{proposition}\label{prop:conf-bos}
The PVA\/ $\mc B_{\mf h}$ is conformal with central charge\/ $0$
and the Virasoro vector
\begin{equation}\label{eq:vir-bos}
L = \frac12\sum_{j=1}^N u^j u_j
\end{equation}
(which is independent of the choice of basis).
The generators\/ $a\in\mf h$ of\/ $\mc B_{\mf h}$ have conformal weight\/ $\Delta(a)=1$.
\end{proposition}
\begin{proof}
First, in order to check \eqref{eq:l1}, we compute for $a\in\mf h$, using the left Leibniz rule L4, \eqref{eq:boson} and \eqref{eq:dualb}:
\begin{align*}
[a_\la L]
&= \frac12\sum_{j=1}^N \big(
[a_\la u^j]u_j+(-1)^{p(u_j)}[a_\la u_j]u^j
\big) \\
& = 
\frac12\sum_{j=1}^N \big(
(a|u^j)u_j+(-1)^{p(u_j)}(a|u_j)u^j
\big)\la
= 
a\la
\,.
\end{align*}
Hence, by skewsymmetry,
$[L_\la a]=(\partial+\la)a$,
so that
$L_{(0)} a = \partial a$ and $L_{(1)} a = a$. 
Then \eqref{eq:l1} follows from the fact that $L_{(0)}$ and $L_{(1)}$ are derivations of the product in $\mc B_{\mf h}$ and $\mc B_{\mf h}$ is generated as a differential algebra by $\mf h$.

Next, we compute $[L_\la L]$ using the Leibniz rule L4 and the fact that $\partial$ is a derivation of the product:
\begin{align*}
[L_\la L] 
&= \frac12\sum_{j=1}^N [L_\la u^j] u_j
+ \frac12\sum_{j=1}^N u^j [L_\la u_j]
\\
&= \frac12\sum_{j=1}^N \bigl( (\partial + \la)u^j \bigr) u_j
+ \frac12\sum_{j=1}^N u^j \bigl( (\partial + \la)u_j \bigr) 
\\
&= (\partial + 2\la) L
\,.
\end{align*}
This completes the proof.
\end{proof}

\begin{theorem}\label{thm:coh-bos}
For the free superboson PVA\/ $\mc B_{\mf h}$, we have
$$
H_{\PV}^n(\mc B_{\mf h}, \mc B_{\mf h}) \simeq (S^n(\Pi\mf h))^* \oplus (S^{n+1}(\Pi\mf h))^* 
\,, \qquad n\ge0 \,.
$$
Explicitly, 
an element $\alpha+\beta\in (S^n(\Pi\mf h))^* \oplus (S^{n+1}(\Pi\mf h))^*$
corresponds under this isomorphism to the\/ $n$-cocycle\/ 
$Y \in C_{\PV}^{n}(\mc B_{\mf h}, \mc B_{\mf h})$, 
uniquely defined by
$$
Y_{\la_1,\dots,\la_n}(u)
= \al(u) + \sum_{j=1}^N \be(u\otimes u^j) u_j 
+ \langle\partial+\la_1+\dots+\la_n\rangle 
\,,\qquad u\in\mf h^{\otimes n}
\,.
$$
\end{theorem}
\begin{proof}
First, note that by Lemma \ref{lem:id}(a), every $n$-cochain $Y$
is uniquely determined by its restriction to $\mf h^{\otimes n}$. 
By \eqref{eq:DY} and Theorem \ref{thm:Delta}, every cohomology class in 
$H_{\PV}^n(\mc B_{\mf h}, \mc B_{\mf h})$ has a representative $Y$ such that
$$
\Delta\bigl( Y_{\la_1,\dots,\la_n}(u) \bigr)
= \Delta(Y) = 0 \text{ or } 1 
\,, \qquad u\in\mf h^{\otimes n} \,.
$$
By definition (see \eqref{eq:euler-k}), this means that $Y$ has the form
\begin{equation}\label{eq:coh-bos-3}
\begin{split}
Y_{\la_1,\dots,\la_n}(u) 
&= \al(u) + \sum_{j=1}^N \be(u\otimes u^j) u_j 
+ \sum_{i=1}^n \ga(u \otimes e_i) \la_i
+ \langle\partial+\la_1+\dots+\la_n\rangle
\,,
\end{split}
\end{equation}
for $u\in\mf h^{\otimes n}$ and 
for some linear maps
$$
\al\colon \mf h^{\otimes n} \to \mb F \,,\qquad
\be\colon \mf h^{\otimes (n+1)} \to \mb F \,,\qquad
\ga\colon \mf h^{\otimes n} \otimes \mb F^n \to \mb F \,,
$$
where $\{e_1,\dots,e_n\}$ is the standard basis for $\mb F^n$.
Notice that
$\al$ and $\be$ are uniquely determined from $Y$,
while $\ga$ is determined up to adding $\bar\gamma\otimes\epsilon$,
where $\bar\ga\colon\mf h^{\otimes n}\to\mb F$ is an arbitrary linear map
and $\epsilon\colon\mb F^n\to\mb F$ is the linear map given by
$\epsilon(e_i)=1$ for all $i=1,\dots,n$.

The symmetry conditions \eqref{eq:skew} for $Y$ 
can be translated in terms of the linear maps $\alpha$, $\beta$ and $\gamma$ as follows.
The maps $\alpha$ and $\beta$ are invariant with respect to the usual action of 
the symmetric group $S_n$ on the vector superspace $(\Pi\mf h)^{\otimes n}$
(where we add a minus sign every time we exchange two odd factors):
\begin{equation*}
\alpha\in (S^n(\Pi\mf h))^*
\,,\qquad
\beta\in (S^n(\Pi\mf h)\otimes\Pi\mf h)^*
\,,
\end{equation*}
while $\gamma$ satisfies the $S_n$-equivariance
\begin{equation}\label{eq:gamma0}
\gamma(\sigma(u)\otimes e_{\sigma(i)})=\gamma(u\otimes e_i)
\,, \qquad
\sigma\in S_n \,, \; u \in (\Pi\mf h)^{\otimes n} \,, \; i=1,\dots,n
\,.
\end{equation}
The map $\gamma$
is defined modulo elements of the form $\bar\gamma\otimes\epsilon$, 
where $\bar\ga \in (S^n(\Pi\mf h))^*$.

Using the definition of the differential \eqref{eq:lca-d}, the $\la$-bracket \eqref{eq:boson},
and Lemma \ref{lem:id}(b),
we can write down an explicit formula for $dY$.
If $Y$ is as in \eqref{eq:coh-bos-3}, we have for $a_i\in\mf h$:
\begin{align}
\notag
(&d Y)_{\la_0,\dots,\la_n} (a_0 \otimes\dots\otimes a_n) 
\\ \label{eq:coh-bos-4}
&= \sum_{i=0}^n \sum_{j=1}^N (-1)^{\ga_i} 
\la_i ( a_i | u_j) \, \be(a_0\otimes\stackrel{i}{\check{\dots}}\otimes a_n \otimes u^j) 
 + \langle\partial+\la_1+\dots+\la_n\rangle
\\ \notag
&= \sum_{i=0}^n (-1)^{\bar p(a_i) (\bar p(a_{i+1})+\cdots+\bar p(a_{n}))} 
\be(a_0\otimes\stackrel{i}{\check{\dots}}\otimes a_n\otimes a_i) 
\,\la_i + \langle\partial+\la_1+\dots+\la_n\rangle
\,,
\end{align}
where $\ga_i$ is given by \eqref{eq:gammai}.
For the last equality, we used \eqref{eq:gammai}, \eqref{eq:symmetry}, \eqref{eq:dualb}, and the fact
that all nonzero summands satisfy
$\bar p(Y)=\bar p(\be)+\bar p(u_j)$, 
$\bar p(\beta)=\bar p(a_0)+\dots+\bar p(a_n)$, and $\bar p(u_j)=\bar p(a_i)$.

Let us denote by $A^{n}$, $B^{n}$ and $C^{n}$ the subspaces of 
$C_{\PV}^{n}(\mc B_{\mf h}, \mc B_{\mf h})$ consisting of $n$-cohains $Y$
corresponding to maps $\al$, $\be$ and $\ga$, respectively.
By \eqref{eq:coh-bos-4}, we have
$$
d(A^{n}) = 0 \,, \qquad
d(B^{n}) \subset C^{n+1} \,, \qquad
d(C^{n}) = 0 \,.
$$
Hence,
\begin{equation*}
H^n_{\PV}(\mc B_{\mf h},\mc B_{\mf h})
\simeq
A^{n} \oplus \ker(d\colon B^{n} \to C^{n+1}) \oplus C^{n} / d(B^{n-1})
\,.
\end{equation*}
By definition, $dY=0$ if and only if all coefficients in front of $\la_i$ in the right-hand side 
of \eqref{eq:coh-bos-4} are equal.
This is equivalent to the condition that $\be \in (S^{n+1}(\Pi\mf h))^*$.

Finally, we claim that $d(B^{n}) = C^{n+1}$.
Indeed, denote the standard basis for $\mb F^{n+1}$ by $\{e_0,\dots,e_n\}$.
By \eqref{eq:gamma0} (with $n$ replaced by $n+1$), 
an $S_{n+1}$-equivariant linear map 
$\ga \colon (\Pi\mf h)^{\otimes (n+1)} \otimes \mb F^{n+1}  \to \mb F$
is uniquely determined by the linear map
\begin{equation*}
\beta\in ((S^n\Pi\mf h)\otimes\Pi\mf h)^* \,, \qquad
\beta(u\otimes a) = \gamma(u\otimes a\otimes e_n) \,.
\end{equation*}
Then, by \eqref{eq:coh-bos-4}, 
the element $Y\in B^{n}$ associated to $\be$ maps to $\ga$ under the differential $d$.
This completes the proof of the theorem.
\end{proof}

\begin{remark}\label{rem:coh-bos}
When $\mf h$ is purely even, we can identify the symmetric powers $S^n(\Pi\mf h)$ with the exterior powers $\bigwedge^n \mf h$. In this case, Theorem \ref{thm:coh-bos} was proved in \cite{DSK12}.
\end{remark}

\begin{corollary}\label{cor:coh-bos}
\begin{enumerate}[(a)]
\item
Every Casimir element of the PVA\/ $\mc B_{\mf h}$ is a linear combination of\/ $\tint 1$ and\/ $\tint u_i$
$(i=1,\dots,N)$.

\item
Every derivation of the PVA\/ $\mc B_{\mf h}$ is a linear combination of an inner derivation and derivations of the form
$$
\frac{\partial}{\partial u_i} \,, \quad 
\sum_{n\in\mb Z_+} \Bigl( u^{i(n)} \frac{\partial}{\partial u_j^{(n)}} 
+ (-1)^{\bar p(u_i) \bar p(u_j)} u^{j(n)} \frac{\partial}{\partial u_i^{(n)}} \Bigr) \,,
\qquad 1 \le i \le j \le N \,,
$$
where, as before, $a^{(n)} = \partial^n a$.
\end{enumerate}
\end{corollary}

\subsection{Cohomology of the free superfermion PVA}\label{sec:coh-fer}

Consider now the free superfermion PVA $\mc F_{\mf h}$, defined in Example \ref{ex:fermion-pva}.
Let again $\{u_1,\dots,u_N\}$ be a basis for $\mf h$, which is homogeneous with respect to parity,
and let $\{u^1,\dots,u^N\}$ be its dual basis, so that $(u_i|u^j)=\delta_i^j$. 
Note that \eqref{eq:dualb} still holds, but now 
\begin{equation}\label{eq:symmetry2}
(a|b)=-(-1)^{p(a)} (b|a)
\,,\qquad a,b\in\mf h\,.
\end{equation}

\begin{proposition}\label{prop:conf-fer}
The PVA\/ $\mc F_{\mf h}$ is conformal with central charge\/ $0$
and the Virasoro vector
\begin{equation}\label{eq:vir-fer}
L = \frac12\sum_{j=1}^N (\partial u^j) u_j
\end{equation}
(which is independent of the choice of basis).
The generators\/ $a\in\mf h$ of\/ $\mc F_{\mf h}$ have conformal weight\/ $\Delta(a)=1/2$.
\end{proposition}
\begin{proof}
The proof is similar to that of Proposition \ref{prop:conf-bos}.
First, by the left Leibniz rule L4, the sesquilinearity L1,
\eqref{eq:fermion}, \eqref{eq:dualb} and \eqref{eq:symmetry2}, we have for $a\in\mf h$:
$$
[a_\la L]
= \frac12\sum_{j=1}^N \big(\lambda(a|u^j)u_j+(-1)^{p(u_j)}(a|u_j)\partial u^j\big)
=\frac12(\lambda-\partial)a
\,.
$$
Hence, from skewsymmetry, $[L_\lambda a]=(\partial+\frac12\lambda)a$,
which implies \eqref{eq:l1} and $\Delta(a)=1/2$.

Next, using the Leibniz rule L4 and the sesquilinearity L1, we compute:
\begin{align*}
[L_\la L] 
&= \frac12\sum_{j=1}^N [L_\la (\partial u^j)] u_j
+ \frac12\sum_{j=1}^N (\partial u^j) [L_\la u_j]
\\
&= \frac12\sum_{j=1}^N \Bigl( (\partial + \la) \Bigl(\partial + \frac12 \la\Bigr) u^j \Bigr) u_j
+ \frac12\sum_{j=1}^N (\partial u^j) \Bigl( \Bigl(\partial + \frac12 \la\Bigr) u_j \Bigr) 
\\
&= (\partial + 2\la) L
\,.
\end{align*}
For the last equality, we used that $\partial$ is a derivation of the product and that
\begin{equation}\label{eq:vir-bos2}
\sum_{j=1}^N u^j u_j = 0 \,.
\end{equation}
To prove \eqref{eq:vir-bos2}, notice that its left side is independent of the choice of basis.
If we start with the basis $\{u^j\}_{j=1,\dots,N}$, then its dual basis is
$\{(-1)^{p(u_j)+1} u_j\}_{j=1,\dots,N}$. Hence,
$$
\sum_{j=1}^N u^j u_j = \sum_{j=1}^N (-1)^{p(u_j)+1} u_j u^j =
- \sum_{j=1}^N u^j u_j \,,
$$
which completes the proof.
\end{proof}

Let $\Phi\colon\mf h\to\mf h^*$ be the vector superspace isomorphism given by
$\Phi(v)(a) = (v|a)$ for $v,a\in\mf h$.
The Lie superalgebra $\mf{gl}(\mf h)$ can be identified with $\mf h\otimes\mf h^*$, so that its standard representation on $\mf h$ is given by $(u\otimes \varphi) \cdot a = \varphi(a) u$. After applying $\Phi$, we have $\mf{gl}(\mf h) \simeq \mf h\otimes\mf h$ and it acts on $\mf h$ by 
$(u\otimes v) \cdot a = (v|a) u$. Let us embed the symmetric square $S^2 \mf h$ into $\mf h\otimes\mf h$ via the map 
$$
uv \mapsto u \otimes v + (-1)^{p(u) p(v)} v \otimes u 
\,, \qquad u,v\in\mf h \,.
$$
Then $S^2 \mf h$ is a subalgebra of the Lie superalgebra $\mf{gl}(\mf h)$, isomorphic to 
\begin{equation}\label{eq:spo}
\mf{spo}(\mf h) = \bigl\{ x\in\mf{gl}(\mf h) \,\big|\, 
(x \cdot a | b) + (-1)^{p(x) p(a)} (a | x \cdot b) = 0 \;\;\text{for all}\;\; a,b\in\mf h \bigr\}
\end{equation}
(see \cite{K77}).
We shall need the following lemma.

\begin{lemma}\label{lem:spo-fer}
The map that sends\/ $x \in S^2 \mf h \subset\mc F_{\mf h}$ 
to\/ $x_{(0)} \in\End\mc F_{\mf h}$ corresponds,
via the identification $S^2\mf h\simeq \mf{spo}(\mf h)$,
to a representation of the Lie superalgebra\/ $\mf{spo}(\mf h)$ on\/ $\mc F_{\mf h}$, 
which coincides with the standard representation when restricted to\/ $\mf h$.
\end{lemma}
\begin{proof}
Take $x=uv \in S^2 \mf h \subset \mc F_{\mf h}$, where $u,v\in\mf h$.
Using again the right Leibniz rule L4', the sesquilinearity L1, and
\eqref{eq:fermion}, we see that
\begin{align*}
[x_\la a] &= [(uv)_\la a]
= (e^{\partial\partial_\la} u) (v | a)
+ (-1)^{p(u) p(v)} (e^{\partial\partial_\la} v) (u | a) 
\\
&= (v | a) u + (-1)^{p(u) p(v)} (u | a) v = x \cdot a
\end{align*}
is precisely the standard action of $x\in\mf{spo}(\mf h)$ on $a\in\mf h$.
Then, in particular, $x_{(0)} a = x \cdot a$ as claimed. 

To show that we have a representation of $\mf{spo}(\mf h)$ on $\mc F_{\mf h}$, notice that $[x_{(0)}, y_{(0)}] = (x_{(0)} y)_{(0)}$, so we only need to prove that $x_{(0)} y = [x,y]$
for $x,y \in\mf{spo}(\mf h)$. 
Taking $y=ab$ for $a,b\in\mf h$, we find:
\begin{align*}
[x_\la y] &= [x_\la (ab)]
= [x_\la a] b + (-1)^{p(x)p(a)} a [x_\la b]
\\
&= (x \cdot a) b + (-1)^{p(x)p(a)} a (x \cdot b)
= [x,y]
\,,
\end{align*}
which completes the proof.
\end{proof}

Now we can determine the variational PVA cohomology of the free superfermion PVA $\mc F_{\mf h}$.

\begin{theorem}\label{thm:coh-fer}
We have
$$
H_{\PV}^0(\mc F_{\mf h}, \mc F_{\mf h}) \simeq \mb F \,\tint 1 \,, \qquad
H_{\PV}^n(\mc F_{\mf h}, \mc F_{\mf h}) = 0 \,, \qquad n\ge1 \,.
$$
\end{theorem}
\begin{proof}
By Lemma \ref{lem:id}(a), any $n$-cochain $Y$
is uniquely determined by its restriction to $\mf h^{\otimes n}$. 
By \eqref{eq:DY} and Theorem \ref{thm:Delta}, every cohomology class in 
$H_{\PV}^n(\mc F_{\mf h}, \mc F_{\mf h})$ has a representative $Y$ such that
\begin{equation}\label{eq:coh-fer}
\Delta( Y_{\la_1,\dots,\la_n}(u) )
+ \frac{n}2 = \Delta([Y]) = 0 \text{ or } 1 
\,, \qquad u\in\mf h^{\otimes n} \,.
\end{equation}
Since $\Delta(Y_{\la_1,\dots,\la_n}(u)) \ge 0$
(see \eqref{eq:euler-k}), we obtain that $Y$ is trivial for $n\ge 3$.
We will consider separately the three cases $n=0,1$ and $2$.

For $n=0$, a $0$-cocycle $Y\in \mc F_{\mf h} / \partial \mc F_{\mf h}$ 
is the same as a Casimir element (see Theorem \ref{thm:lowcoho2}(a)).
Since $\Delta(Y)=0$ or $1$, we have that $Y\in\mb F \tint 1$ or $Y=\tint x$ for some $x\in S^2 \mf h$.
But, by Lemma \ref{lem:spo-fer}, $x_{(0)} \ne 0$ for every nonzero $x\in S^2 \mf h$.
Hence, any Casimir element is a scalar multiple of $\tint 1$.
This proves that 
$H_{\PV}^0(\mc F_{\mf h}, \mc F_{\mf h}) \simeq \mb F \,\tint 1$,
as claimed.

For $n=1$, formula \eqref{eq:coh-fer} implies that 
$Y_\la(u) = \be(u)$ for $u\in\mf h$, where $\be\colon\mf h \to \mf h$ is a linear map.
A simple calculation using 
\eqref{eq:lca-d}, \eqref{eq:fermion} and Lemma \ref{lem:id}(b) gives
\begin{equation}\label{eq:coh-fer2}
\begin{split}
(dY&)_{\la_0,\la_1}(a_0 \otimes a_1) 
= (-1)^{\ga_0} (a_0 | \beta(a_1))
+ (-1)^{\ga_1} (a_1 | \beta(a_0))
+ \langle \la_0+\la_1 \rangle
\\
&= 
(-1)^{p(a_1)}
\big(
(\beta(a_0)|a_1)+(-1)^{p(\beta)p(a_0)}(a_0|\beta(a_1))
\big)
+ \langle \la_0+\la_1 \rangle
\,,
\end{split}
\end{equation}
where $a_i\in\mf h$ and the $\ga_i$ are given by \eqref{eq:gammai}.
Hence, $dY=0$ if and only if $\beta\in\mf{spo}(\mf h)$.
However, all elements of $\mf{spo}(\mf h)$ correspond to coboundaries,
because they give inner derivations, due to Lemma \ref{lem:spo-fer}.
We conclude that 
$H_{\PV}^1(\mc F_{\mf h}, \mc F_{\mf h}) = 0$.

Finally, consider the case $n=2$. 
In this case \eqref{eq:coh-fer} implies that 
any $2$-cocycle $Y$ is equivalent to a cocycle of the form
$$
Y_{\la_0,\la_1}(a_0 \otimes a_1) = \ga(a_0 \otimes a_1)
+\langle\lambda_0+\lambda_1\rangle
\,, \qquad a_0, a_1 \in\mf h \,,
$$
for some linear map $\ga\colon S^2(\Pi\mf h) \to\mb F$.
All such maps are cocycles.
By \eqref{eq:coh-fer2},
the vector space of coboundaries is isomorphic to $\mf{gl}(\mf h)/\mf{spo}(\mf h)$.
Since
$$
\dim S^2(\Pi\mf h)
=
\dim\mf{gl}(\mf h)-\dim\mf{spo}(\mf h)
\,,
$$
we conclude that $H_{\PV}^2(\mc F_{\mf h}, \mc F_{\mf h}) = 0$.
This completes the proof.
\end{proof}

\begin{corollary}\label{cor:coh-fer}
\begin{enumerate}[(a)]
\item
Every Casimir element of the PVA\/ $\mc F_{\mf h}$ is a scalar multiple of~$\,\tint 1$.
\item
Every derivation of the PVA\/ $\mc F_{\mf h}$ is inner.
\item
Any first-order deformation of\/ $\mc F_{\mf h}$ that preserves the product and the\/ $\mb F[\partial]$-module structure is trivial.
\end{enumerate}
\end{corollary}

\subsection{Cohomology of the affine PVA}\label{coh:coh-aff}

Consider now the affine PVA $\mc V^k_{\mf g}$ at level $k\in\mb F$ from Example \ref{ex:affine-pva},
where $\mf g$ is a finite-dimensional Lie algebra with a nondegenerate symmetric invariant bilinear form $(\cdot|\cdot)$.
Let $\{u_1,\dots,u_N\}$ be a basis for $\mf g$, and $\{u^1,\dots,u^N\}$ be its dual basis with respect to $(\cdot|\cdot)$.

\begin{proposition}\label{prop:conf-aff}
For\/ $k\ne0$, the PVA\/ $\mc V^k_{\mf g}$ is conformal with central charge\/ $0$
and the Virasoro vector
\begin{equation}\label{eq:vir-aff}
L = \frac1{2k} \sum_{j=1}^N u^j u_j
\end{equation}
(which is independent of the choice of basis).
The generators\/ $a\in\mf g$ of\/ $\mc V^k_{\mf g}$ have conformal weight\/ $\Delta(a)=1$.
\end{proposition}
\begin{proof}
Using the left Leibniz rule L4, \eqref{eq:current} and \eqref{eq:dualb},
we compute for $a\in\mf g$:
$$
[a_\lambda L]
=
\frac1{2k}
\sum_{j=1}^N\big(
[a,u^j]u_j+u^j[a,u_j]+k\lambda(a|u^j)u_j+k\lambda(a|u_j)u^j
\big)
=
\lambda a
\,.
$$
In the second equality, we used that the Casimir element $\sum_{j=1}^N u^j u_j \in S^2 \mf g$ is invariant under the adjoint action of $\mf g$:
$$
\sum_{j=1}^N \bigl( [a,u^j] u_j + u^j [a,u_j] \bigr) = 0 \,.
$$
Hence, by skewsymmetry, $[L_\lambda a]=(\partial+\lambda a)$.
The rest of the proof is exactly the same as for Proposition \ref{prop:conf-bos}.
\end{proof}

\begin{theorem}\label{thm:coh-aff}
For any finite-dimensional Lie algebra\/ $\mf g$
with a nondegenerate symmetric invariant bilinear form,
and any nonzero level\/ $k\in\mb F$, we have
$$
H_{\PV}^n(\mc V^k_{\mf g}, \mc V^k_{\mf g}) 
\simeq H^n(\mf g, \mb F) \oplus H^{n+1}(\mf g, \mb F)
\,, \qquad n\ge0 \,.
$$
Explicitly, 
an element\/ $[\alpha]+[\beta]\in H^n(\mf g, \mb F)\oplus H^{n+1}(\mf g, \mb F)$
corresponds under this isomorphism to the\/ $n$-cocycle\/ 
$Y \in C_{\PV}^{n}(\mc V^k_{\mf g}, \mc V^k_{\mf g})$, 
uniquely defined by
$$
Y_{\la_1,\dots,\la_n}(u)
= \al(u) + \sum_{j=1}^N \be(u\otimes u^j) u_j 
+ \langle\partial+\la_1+\dots+\la_n\rangle 
\,,\qquad u\in\mf g^{\otimes n}
\,.
$$
\end{theorem}
\begin{proof}
Since $\Delta(a)=1$ for all $a\in\mf g$,
the portion of the proof of Theorem \ref{thm:coh-bos} that does not involve the differential $d$ translates verbatim to the current case.
In particular, every $n$-cocycle $Y$ is equivalent to one of the form
\begin{equation}\label{equation}
Y_{\la_1,\dots,\la_n}(u) 
= \al(u) + \sum_{j=1}^N \be(u\otimes u^j) u_j 
+ \sum_{i=1}^n \ga(u \otimes e_i) \la_i
+ \langle\partial+\la_1+\dots+\la_n\rangle
\,,
\end{equation}
for $u\in\mf g^{\otimes n}$ and some linear maps
$$
\al\colon \textstyle\bigwedge\nolimits^n \mf g \to \mb F \,,\qquad
\be\colon (\textstyle\bigwedge\nolimits^n \mf g)\otimes\mf g \to \mb F \,,\qquad
\ga\colon \bigl( (\Pi\mf g)^{\otimes n} \otimes \mb F^n \bigr)^{S_n} \to \mb F \,,
$$
where $\{e_1,\dots,e_n\}$ is the standard basis for $\mb F^n$.

Denote by $A^{n}$, $B^{n}$ and $C^{n}$ the subspaces of 
$C_{\PV}^{n}(\mc V^k_{\mf g}, \mc V^k_{\mf g})$ consisting of $n$-cohains $Y$
corresponding to maps $\al$, $\be$ and $\ga$, respectively.
Using \eqref{eq:lca-d}, \eqref{eq:current}, and Lemma \ref{lem:id}(b), we see that
$$
d(A^{n}) \subset A^{n+1} \,, \qquad
d(B^{n}) \subset B^{n+1} \oplus C^{n+1} \,, \qquad
d(C^{n}) \subset C^{n+1} \,.
$$
Notice that the action of $d$ on $A^{n}$ corresponds to applying 
the Lie algebra cohomology differential to $\al$,
viewed as an $n$-cochain for the Lie algebra $\mf g$ with coefficients in $\mb F$. 
As in the proof of Theorem \ref{thm:coh-bos},
this gives the summand $H^n(\mf g,\mb F)$ inside $H^n_{\PV}(\mc V^k_{\mf g},\mc V^k_{\mf g})$.

We next concentrate on the subcomplex $B^\bullet\oplus C^\bullet$.
For $Y\in B^{n}$, 
denote by
$d_1Y\in B^{n+1}$ and $d_2Y\in C^{n+1}$
the projections of $dY$ on $B^{n+1}$ and $C^{n+1}$, respectively.
Notice that $d_2\colon B^{n}\to C^{n+1}$
coincides with the map $d|_{B^{n}}$ from the proof of Theorem \ref{thm:coh-bos}.
In particular, $d_2$ is surjective.
As a consequence, we can assume that 
\begin{equation}\label{eq:coh-aff}
Y_{\la_1,\dots,\la_n}(u) 
= \sum_{j=1}^N \be(u\otimes u^j) u_j 
+ \langle\partial+\la_1+\dots+\la_n\rangle
\,,\qquad u\in\mf g^{\otimes n}
\,,
\end{equation}
where $d_2Y=0$.

Recall from the proof of Theorem \ref{thm:coh-bos} that
the condition $d_2Y=0$ is equivalent to the skewsymmetry of $\beta$,
i.e., $\beta\in (\bigwedge^{n+1}\mf g)^*$.
Then for $Y$ as in \eqref{eq:coh-aff} and $a_i\in\mf g$, we find
\begin{align*}
(d_1Y&)_{\lambda_0,\dots,\lambda_n}(a_0\otimes\dots\otimes a_n)
=
\sum_{i=0}^n(-1)^{n+i+1}\sum_{\ell=1}^N
\beta(a_0\otimes\stackrel{i}{\check{\dots}}\otimes a_n\otimes u^\ell)\, [a_i,u_\ell] \\
& +
\sum_{0\leq i<j\leq n}(-1)^{n+i+j+1}\sum_{\ell=1}^N
\beta([a_i,a_j]\otimes a_0\otimes\stackrel{i}{\check{\dots}}\stackrel{j}{\check{\dots}}\otimes a_n\otimes u^\ell)u_\ell
\,.
\end{align*}
Using that
$$
\sum_{\ell=1}^N
u^\ell\otimes [a,u_\ell]
=
-
\sum_{\ell=1}^N
[a,u^\ell]\otimes u_\ell
\,,\qquad a\in\mf g
\,,
$$
we see that the restriction of $d_1$ to $(\bigwedge^{n+1}\mf g)^*$
coincides with the Lie algebra cohomology differential for $\mf g$ with coefficients in $\mb F$.
This completes the proof.
\end{proof}
\begin{remark}
Theorem \ref{thm:coh-aff} can be easily generalized to the superalgebra case.
In the special case when $\mf g$ is an abelian Lie superalgebra, 
we recover Theorem \ref{thm:coh-bos}.
\end{remark}
\begin{remark}
Assume that $\mf g$ is a finite-dimensional simple Lie algebra.
Note that $\beta$ in \eqref{equation} can also be viewed as a linear map $\beta\colon \bigwedge^n\mf g\to\mf g$,
and the differential $d_1$ coincides with the Lie algebra cohomology differential with coefficients in $\mf g$.
Using that $H^n(\mf g,\mf g)=0$,
we can therefore assume that $\beta=0$ in equation \eqref{equation}.
In this way one can prove directly that 
$$
H^n_{\PV}(\mc V^k_{\mf g},\mc V^k_{\mf g})
\simeq
H^n_{\LC}(\overline{\cur}\,\mf g,\mb F)
\,,
$$
which is consistent with Proposition \ref{pbarcurf}.
However, this isomorphism does not hold when $\mf g$ is abelian.
\end{remark}

\begin{corollary}\label{cor:coh-aff}
Let\/ $\mf g$ be a finite-dimensional Lie algebra with a nondegenerate symmetric invariant bilinear form,
and\/ $k\in\mb F$ be nonzero.
\begin{enumerate}[(a)]
\item
Every Casimir element of the PVA\/ $\mc V^k_{\mf g}$ has the form\/ $\tint (\al 1+c)$,
where\/ $\al\in\mb F$ and\/ $c\in \cent(\mf g)$, the center of the Lie algebra\/ $\mf g$.
\item
Every derivation of the PVA\/ $\mc V^k_{\mf g}$ is a sum of an inner derivation and a derivation that acts on the generators of\/ $\mc V^k_{\mf g}$ as\/ $D_1(a) + D_2(a)$, $a\in\mf g$, where\/
$D_1 \colon\mf g\to\mb F$ is such that\/
$D_1([\mf g,\mf g]) = 0$
and
$D_2 \colon\mf g\to\mf g$
is such that\/
$$
D_2[a,b] = [D_2a,b]+[a,D_2b] \,, \quad
(D_2a|b) + (a|D_2b) = 0 \,, \qquad 
a,b\in\mf g \,,
$$
i.e., $D_2 \in 
\Der(\mf g) \cap \mf o(\mf g)$.
\end{enumerate}
\end{corollary}
\begin{proof}
This follows from the definitions and the proof of Theorem \ref{thm:coh-aff}.
\end{proof}

\begin{remark}\label{cor:coh-12}
Let $\mf g$ be a finite-dimensional Lie algebra with a nondegenerate symmetric invariant bilinear form.
Then Corollary \ref{cor:coh-aff} and Theorem \ref{thm:coh-aff} imply the following well-known isomorphisms:
\begin{align*}
H^1(\mf g,\mb F) &= \bigl( \mf g / [\mf g,\mf g] \bigr)^* \simeq \cent(\mf g) 
\,, \\
H^2(\mf g,\mb F) &\simeq \bigl( \Der(\mf g) \cap \mf o(\mf g) \bigr) \big/ \Inder(\mf g)
\,.
\end{align*}
Explicitly, $c\in \cent(\mf g)$ corresponds to the $1$-cocycle $\al(a)=(a|c)$, and 
$D \in \Der(\mf g) \cap \mf o(\mf g)$ corresponds to the $2$-cocycle 
$\be(a \otimes b)=(Da|b)$, for $a,b\in\mf g$.
\end{remark}

\begin{corollary}\label{cor:coh-aff-sim}
Let\/ $\mf g$ be a finite-dimensional simple Lie algebra, and\/ $k\in\mb F$ be nonzero.
\begin{enumerate}[(a)]
\item
Every Casimir element of the PVA\/ $\mc V^k_{\mf g}$ is a scalar multiple of~$\,\tint 1$.
\item
Every derivation of the PVA\/ $\mc V^k_{\mf g}$ is inner.
\item
Up to equivalence, any first-order deformation of\/ $\mc V^k_{\mf g}$,
which preserves the product and the\/ $\mb F[\partial]$-module structure,
corresponds to a scalar multiple of the\/ $2$-cocycle given by\/
$Y_{\la_1,\la_2}(a_1 \otimes a_2) = \la_1 (a_1 | a_2)+ \langle \partial+\la_1+\la_2\rangle$ 
for\/ $a_1,a_2\in\mf g$.
\end{enumerate}
\end{corollary}

\begin{remark}\label{rem:barcurg}
Let $\mf g$ be a finite-dimensional simple Lie algebra.
Any Lie algebra cocycle $\be\in\bigl( \bigwedge\nolimits^{n-1} \mf g^* \bigr)^{\mf g}$
gives a PVA cochain $Z \in C_{\PV}^{n-1}(\mc V^k_{\mf g}, \mc V^k_{\mf g})$
defined by
$$
Z_{\la_1,\dots,\la_{n-1}}(a_1 \otimes\dots\otimes a_{n-1}) 
= \be(a_1 \wedge\dots\wedge a_{n-1}) L 
+ \langle \partial+\la_1+\dots+\la_{n-1} \rangle
\,, \qquad a_i\in\mf g\,.
$$
Using that $d\be=0$ in the Lie algebra cohomology complex and $[a_\la L] = \la a$ for $a\in\mf g$,
we find that
$dZ=Y$ is given by equation \eqref{curgmv7} from Proposition \ref{prop:barcurg}. This proves that $dY=0$ 
in both the variational PVA and LCA cohomology complexes. 
We claim that $Y$ is not exact in LCA cohomology. Indeed, if $Y=dX$
for some $X \in C_{\LC}^{n-1}(\overline{\cur}\,\mf g, \overline{\cur}\,\mf g)$, then
$\Delta(X)=\Delta(Y)=2$
implies that $X$ has the form
\begin{align*}
X_{\la_1,\dots,\la_{n-1}}(u) 
= \partial f(u)
+ \sum_{i=1}^{n-1} g_i(u) \la_i
+ \langle \partial+\la_1+\dots+\la_{n-1} \rangle
\,, 
\qquad u \in \mf g^{\otimes (n-1)} \,,
\end{align*}
for some linear maps 
$f,g_i\colon \mf g^{\otimes (n-1)} \to \mf g$.
Setting $\partial = -\la_1-\dots-\la_{n-1}$, we can assume that $f=0$.
Then it is straightforward to compute $dX$ and see that $dX \ne Y$.
\end{remark}

\subsection{Cohomology of the Virasoro PVA}\label{sec:coh-vir}

Recall the Virasoro PVA $\PVir^c$ of central charge $c\in\mb F$, defined in Example \ref{ex:virasoro-pva}. Obviously, it is conformal with Virasoro vector $L$ and $\Delta(L)=2$.

\begin{theorem}\label{thm:coh-vir}
For every central charge\/ $c\in\mb F$, we have{\rm:}
$$
H_{\PV}^n(\PVir^c, \PVir^c) \simeq 
H^n_{\LC}(\bar R^\vir,\mb F) = 
\begin{cases} \, \mb F \, [Y^n] \,, \quad\text{for } \; n=0,2,3, 
\\ \, 0 \,, \quad\;\;\; \text{otherwise,}
\end{cases}
$$
where
\begin{align*}
Y^0 &= 1 + \langle \partial\rangle = \tint 1 \,,
\\
Y^2_{\la_1,\la_2}(L \otimes L) 
&= \la_1^3 + \langle \partial+\la_1+\la_2\rangle \,,
\\
Y^3_{\la_1,\la_2,\la_3}(L \otimes L \otimes L) 
&= (\la_1-\la_2)(\la_1-\la_3)(\la_2-\la_3) + \langle \partial+\la_1+\la_2+\la_3\rangle \,.
\end{align*}
\end{theorem}
\begin{proof}
By \eqref{eq:DY} and Theorem \ref{thm:Delta}, every cohomology class in 
$H_{\PV}^n(\PVir^c, \PVir^c)$ has a representative $Y$ such that
\begin{equation}\label{eq:coh-vir}
\Delta( Y_{\la_1,\dots,\la_{n}}(L^{\otimes n}) )
- n = \Delta(Y) = 0 \text{ or } 1 
\,,
\end{equation}
hence
\begin{equation}\label{eq:coh-vir1}
\Delta( Y_{\la_1,\dots,\la_{n}}(L^{\otimes n}) )
\leq n+1
\,.
\end{equation}
Recall that, as a differential algebra,
$$
\PVir^c
=
\mb F\bigl[L^{(i)}
\,\big|\, i\in\mb Z_+\bigr] 
\,, \qquad L^{(i)} = \partial^i L
\,,
$$
and $\Delta(L^{(i)}) = i+2$. 
Moreover, $\Delta(ab) = \Delta(a) + \Delta(b)$.

For $n=0$, we have $Y \in \PVir^c / \partial \PVir^c$, and \eqref{eq:coh-vir} implies $\Delta(Y) = 0$ or $1$. Hence, $Y$ is a scalar multiple of $Y^0 = \tint 1$, which is a Casimir element for $\PVir^c$.

Suppose now that $n\ge1$. Note that, by \eqref{eq:skew},
$$
Y_{\la_1,\dots,\la_{n}}(L^{\otimes n}) 
\in \PVir^c[\la_1,\dots,\la_{n}] / 
\langle \partial+\la_1+\cdots+\la_{n} \rangle
$$
is skewsymmetric with respect to $\la_1,\dots,\la_{n}$. Under the isomorphism
$$
\PVir^c[\la_1,\dots,\la_{n}] / 
\langle \partial+\la_1+\cdots+\la_{n} \rangle
\simeq \PVir^c[\la_1,\dots,\la_{n-1}] \,,
$$
we can replace $\la_n$ by $-\partial-\la_1-\cdots-\la_{n-1}$
and assume that 
$$
Y_{\la_1,\dots,\la_{n}}(L^{\otimes n})
\in \PVir^c[\la_1,\dots,\la_{n-1}]
$$
is independent of $\la_n$, and that it is skewsymmetric with respect to $\la_1,\dots,\la_{n-1}$. 
Every skewsymmetric polynomial of $\la_1,\dots,\la_{n-1}$ is divisible by $\prod_{1\le i<j \le n-1} (\la_i-\la_j)$; hence, has degree $\ge \binom{n-1}{2}$. 

Since $\binom{n-1}{2} \ge n-2$ for all $n\ge1$, we see from \eqref{eq:coh-vir1} that $Y$ has the form
\begin{equation}\label{eq:coh-vir-2}
\begin{split}
Y_{\la_1,\dots,\la_{n}}&(L^{\otimes n}) 
= \al(\la_1,\dots,\la_{n}) 1 
+ \be(\la_1,\dots,\la_{n}) L
\\
&+ \ga(\la_1,\dots,\la_{n}) \partial L
+ \langle \partial+\la_1+\cdots+\la_{n} \rangle \,.
\end{split}
\end{equation}
Here $\al$, $\be$ and $\ga$ are polynomials with coefficients in $\mb F$. Note that $\al$ is defined and is skewsymmetric modulo $\la_1+\cdots+\la_{n}$. After replacing $\partial L$ with $-(\la_1+\dots+\la_{n})L$ in \eqref{eq:coh-vir-2}, we can assume that $\ga=0$.
Then $\be$ is uniquely determined and is skewsymmetric in $\mb F[\la_1,\dots,\la_{n}]$. 

As above, we have $\deg\be\ge\binom{n}{2} \ge n-1$, with equality only for $n=2$. But \eqref{eq:coh-vir1} implies 
$$
n+1 \ge 
\deg\be+\Delta(L) \ge (n-1)+2 \,,
$$
which means that if $\be\ne0$,
then $\deg\beta=n-1$, hence $n=2$ and 
$\be(\la_1,\la_2) = (\la_1-\la_2)b$ for some $b\in\mb F$.
Consider the $1$-cochain $Z \in C_{\PV}^1(\PVir^c, \PVir^c)$
defined by $Z_\la(L)=L + \langle \partial+\la \rangle$. Then 
\begin{equation}\label{eq:coh-vir-4}
\begin{split}
(dZ)_{\la_1,\la_2} &(L \otimes L) 
= [L_{\la_1} Z_{\la_2}(L)] - [L_{\la_2} Z_{\la_1}(L)] 
- Z_{\la_1+\la_2}([L_{\la_1} L])
\\
&= (\la_1-\la_2) L + \frac{c}{12} (\la_1^3-\la_1^3) 
+ \langle \partial+\la_1+\la_2 \rangle \,.
\end{split}
\end{equation}
Replacing $Y$ with $Y-b \, dZ$, we can assume that $\be=0$
in \eqref{eq:coh-vir-2}.

We have shown that every $n$-cocycle for $\PVir^c$ is equivalent to a cocycle $Y$ such that
\begin{equation}\label{eq:coh-vir-3}
Y_{\la_1,\dots,\la_{n}}(L^{\otimes n}) 
= \al(\la_1,\dots,\la_{n}) 1 
+ \langle \partial+\la_1+\cdots+\la_{n} \rangle \,.
\end{equation}
We claim that if $Y=dX$ for some $(n-1)$-cochain $X$, then $X$ can be replaced with a cochain that also satisfies \eqref{eq:coh-vir-3} for some $\al$ (with $n$ replaced by $n-1$).
Since $\Delta(Y)=\Delta(X)=0$ or $1$, by the above discussion, we see that $X$ has the form \eqref{eq:coh-vir-2} with $\ga=0$, for some $\al$, $\be$ (again with $n$ replaced by $n-1$). But, as above, $\be$ may be nonzero only when $n-1=2$. In this case,
$$
X_{\la_1,\la_2} (L \otimes L)  
= \al(\la_1,\la_2) 1 
+ (\la_1-\la_2) bL
+ \langle \partial+\la_1+\la_2 \rangle \,,
$$
for some $b\in\mb F$. Comparing this with \eqref{eq:coh-vir-4}, we obtain
that $X-b \, dZ$ has no $L$-term. This proves the claim, as $Y=dX=d(X-b \, dZ)$.

Therefore, we can reduce the coefficients in the cohomology to $\mb F$. By Theorem \ref{prop}(b), we obtain 
$H_{\PV}^n(\PVir^c, \PVir^c) \simeq 
H^n_{\LC}(\bar R^\vir,\mb F)$.
The rest of the proof follows immediately from Proposition \ref{pbarvir} and Remark \ref{rem:label}.
\end{proof}

\begin{corollary}\label{cor:coh-vir}
For every\/ $c\in\mb F$, we have:
\begin{enumerate}[(a)]
\item
Every Casimir element of the PVA\/ $\PVir^c$ is a scalar multiple of~$\,\tint 1$.
\item
Every derivation of the PVA\/ $\PVir^c$ is inner.
\item
Up to equivalence, any first-order deformation of\/ $\PVir^c$,
which preserves the product and the\/ $\mb F[\partial]$-module structure,
corresponds to a scalar multiple of the\/ $2$-cocycle given by\/
$Y_{\la_1,\la_2}(L \otimes L) = \la_1^3 + \langle \partial+\la_1+\la_2\rangle$.
\end{enumerate}
\end{corollary}




\end{document}